\documentclass[a4paper,11pt]{article}%
\usepackage{amssymb}
\usepackage{amsmath}
\usepackage{amsfonts}
\usepackage{graphicx}%
\providecommand{\U}[1]{\protect\rule{.1in}{.1in}}
\newtheorem{theorem}{Theorem}

\newtheorem{corollary}[theorem]{Corollary}

\newtheorem{definition}[theorem]{Definition}

\newtheorem{lemma}[theorem]{Lemma}

\newtheorem{proposition}[theorem]{Proposition}
\newtheorem{remark}[theorem]{Remark}

\newenvironment{proof}[1][Proof]{\noindent\textbf{#1.} }{\ \rule{0.5em}{0.5em}}
\begin{document}

\title{Kusuoka-Stroock gradient bounds for the solution of the filtering equation}
\author{Dan Crisan\thanks{Department of Mathematics, Imperial College London, 180
Queen's Gate, London, SW7 2AZ, United Kingdom}
\and Christian Litterer\thanks{Department of Mathematics, Imperial College London,
180 Queen's Gate, London, SW7 2AZ, United Kingdom}
\and Terry Lyons\thanks{ University of Oxford and Oxford-Man Institute Eagle House,
Walton Well Road, Oxford. OX2 6ED}}
\date{25 September 2013}
\maketitle

\begin{abstract}
We obtain sharp gradient bounds for perturbed diffusion semigroups. In
contrast with existing results, the perturbation is here random and the bounds
obtained are pathwise. Our approach builds on the classical work of Kusuoka
and Stroock \cite{kusuoka, k-s-0,k-s-1,k-s}, and extends their program
developed for the heat semi-group to solutions of stochastic partial
differential equations. The work is motivated by and applied to nonlinear
filtering. The analysis allows us to derive pathwise gradient bounds for the
un-normalised conditional distribution of a partially observed signal. It uses
a pathwise representation of the perturbed semigroup in the spirit of
classical work by Ocone \cite{ocone}. The estimates we derive have sharp small
time asymptotics.\medskip

\medskip


\end{abstract}

\section{Introduction}

In the eighties, Kusuoka and Stroock \cite{kusuoka, k-s-0,k-s-1,k-s} analysed
the smoothness properties of the (perturbed) semigroup associated to a
diffusion process. More precisely, let $\left(  \Omega,\mathcal{F},P\right)  $
be a probability space on which we have defined a $d_{1}$-dimensional standard
Brownian motion $B$ and $X^{x}=\{X_{t}^{x},\ t\geq0\}, $ $x\in{\mathbb{R}}%
^{N}$ be the stochastic flow%

\begin{equation}
X_{t}^{x}=x+\int_{0}^{t}V_{0}(X_{s}^{x})ds+\sum_{i=1}^{d_{1}}\int_{0}^{t}%
V_{i}(X_{s}^{x})\circ dB_{s}^{i},\quad t\geq0, \label{eq:dan:1:2:20}%
\end{equation}
where the vector fields $\left\{  V_{i},~\ i=0,...,d_{1}\right\}  $ are
$C_{b}^{\infty},$ by which me mean they are smooth and bounded with bounded
derivatives of all orders$,$ and the stochastic integrals in
(\ref{eq:dan:1:2:20}) are of Stratonovich type. The corresponding perturbed
diffusion semigroup is then given by%

\[
(P_{t}^{c}\varphi)(x)={\mathbb{E}}\left[  \varphi(X_{t}^{x})\exp\left(
\int_{0}^{t}c\left(  X_{s}^{x}\right)  ds\right)  \right]  ,\quad t\geq0,\quad
x\in\mathbb{R}^{N},
\]
where $c\in C_{b}^{\infty}\left(  {\mathbb{R}}^{N}\right)  $ and
$\varphi:\mathbb{R}^{N}\rightarrow\mathbb{R}$ is an arbitrary bounded
measurable function. The vector fields $\left\{  V_{i},~\ i=0,...,d_{1}%
\right\}  $ are assumed to satisfy Kusuoka's so-called UFG condition. This
condition states that the $C_{b}^{\infty}\left(  \mathbb{R}^{N}\right)
-$module $\mathcal{W}$ generated by the vector fields $\left\{  V_{i}%
,~\ i=1,...,d_{1}\right\}  $ within the Lie algebra generated by $\left\{
V_{i}~~\ i=0,...,d_{1}\right\}  $ is finite dimensional. In particular, the
condition does not require that the vector space $\{W(x)|W\in\mathcal{W}\}$ is
isomorphic to ${\mathbb{R}}^{N}$ for all $x\in{\mathbb{R}}^{N}$. Hence, in
this sense, the UFG condition is weaker than the uniform H\"{o}rmander condition.

\bigskip

Kusuoka, Stroock prove that, under the UFG condition, $P_{t}^{c}\varphi$ is
differentiable in the direction of any vector field $W$ belonging to
$\mathcal{W}$. Moreover, they deduce sharp gradient bounds of the following
form: Given vector fields $W_{i}\in{\mathcal{W}}$, $i=1,...,m+n$ there exist
constants $C>0$, $l>0$ such that
\begin{equation}
\Vert W_{1}\ldots W_{m}P_{t}^{c}(W_{m+1}\ldots W_{m+n}\varphi)\Vert_{p}\leq
Ct^{-l}\Vert\varphi\Vert_{p},\ \ \label{typeofbounds}%
\end{equation}
holds for any $\varphi\in C_{0}^{\infty}\left(  {\mathbb{R}}^{N}\right)  $,
$t\in(0,1]$ and $p\in\left[  1,\infty\right]  $. In fact, the constant $l$
depends explicitly on the vector fields $W_{i}$, $i=1,...,m+n$ and the small
time asymptotics are sharp. In this paper we deduce a similar result for the
randomly perturbed semigroup. More precisely, let
\[
Y=\left\{  \left(  Y_{t}^{i}\right)  _{i=1}^{d_{2}},t\geq0\right\}
\]
be a $d_{2}$-dimensional standard Brownian motion independent of $X$, and define%

\begin{equation}
\rho_{t}^{Y\left(  \omega\right)  }(\varphi)(x)={\mathbb{E}}\left[  \left.
\varphi(X_{t}^{x})Z_{t}^{x}\right\vert \mathcal{Y}_{t}\right]  \left(
\omega\right)  ,\text{~}t\geq0,~x\in\mathbb{R}^{N}, \label{processrho}%
\end{equation}
where $Z^{x}=\left\{  Z_{t}^{x},t\geq0\right\}  $, $x\in\mathbb{R}^{N}$ is the
stochastic process
\begin{equation}
Z_{t}^{x}=\exp\left(  \sum_{i=1}^{d_{2}}\int_{0}^{t}h_{i}\left(  X_{s}%
^{x}\right)  dY_{s}^{i}-\frac{1}{2}\sum_{i=1}^{d_{2}}\int_{0}^{t}h_{i}\left(
X_{s}^{x}\right)  ^{2}ds\right)  ,\text{~}t\geq0,~x\in\mathbb{R}^{N},
\label{zmartingale}%
\end{equation}
$h_{i}\in C_{b}^{\infty}\left(  \mathbb{R}^{N}\right)  $, $i=1,...,d_{2}$ and
$\varphi$ is an arbitrary bounded measurable function on $\mathbb{R}^{N}$.
Then we prove in the following that for the mapping $x\longrightarrow\rho
_{t}^{Y\left(  \omega\right)  }(\varphi)(x),$ there exists a $P$-almost surely
finite random variable $\omega\rightarrow C\left(  \omega\right)  $ such that
with $l$ the explicit constant in $\left(  \ref{typeofbounds}\right)  $ we
have
\begin{equation}
\Vert W_{1}\ldots W_{m}\rho_{t}^{Y\left(  \omega\right)  }\left(
W_{m+1}\ldots W_{m+n}\varphi\right)  \Vert_{p}\leq C\left(  \omega\right)
t^{-l}\Vert\varphi\Vert_{p}, \label{typeofbounds2}%
\end{equation}

for any $\varphi\in C_{0}^{\infty}\left(  \mathbb{R}^{N}\right)  ,$
$t\in(0,1],p\in\left[  1,\infty\right]  $.

\bigskip

We are interested in this particular perturbation as it provides the
Feynman-Kac representation for solutions of linear parabolic stochastic
partial differential equations (SPDEs)\footnote{We expect the methodology
presented here can be \ extended to handle a wider class of random
perturbations. We chose this particular perturbation because the corresponding
randomly perturbed semigroup provides the Feynman-Kac representation for the
solution of the filtering problem. See the Kallianpur-Striebel formula
(\ref{kallstrie}) below.}. More precisely, let $\rho^{x}=\{\rho_{t}%
^{x},\ t\geq0\},$ $x\in{\mathbb{R}}^{N}$ \ be the measure valued process
defined on the probability space $\left(  \Omega,\mathcal{F},P\right)  $ by
the formula
\[
\left(  \rho_{t}^{x}\left(  \omega\right)  \right)  (\varphi)=\rho
_{t}^{Y\left(  \omega\right)  }(\varphi)(x),
\]
where $\varphi$ is an arbitrary Borel measurable function. Then $\rho^{x}$ is
the solution of the following linear parabolic SPDE \ (written here in its
weak form):
\begin{align}
d\rho_{t}^{x}(\varphi)  &  =\rho_{t}^{x}(A\varphi)dt+\sum_{k=1}^{d_{2}}%
\rho_{t}^{x}(h_{k}\varphi)dY_{t}^{k}\label{weakks}\\
\rho_{0}^{x}  &  =\delta_{x}.\nonumber
\end{align}
Here, $\delta_{x}$ is Dirac delta distribution centered at $x\in\mathbb{R}%
^{N}$, $A=V_{0}+\frac{1}{2}\sum_{i=1}^{d_{1}}V_{i}^{2}\;$is the infinitesimal
generator of $X$, and $\varphi$ is a suitably chosen test function. Equation
(\ref{weakks}) is called the Duncan-Mortensen-Zakai equation (cf.
\cite{d,m,z}). It plays a central r\^{o}le in nonlinear filtering: The
normalised solution of (\ref{weakks})\ gives the conditional distribution of a
partially observed stochastic process. We give details of this intrinsic
connection in the second section.

Let us finally note that for a fixed $x\in{\mathbb{R}}^{N}$, and any suitably
chosen test function $\varphi$, the application $Y\left(  \omega\right)
\longrightarrow\rho_{t}^{Y\left(  \omega\right)  }(\varphi)(x)$ is a (locally)
Lipschitz continuous function as defined on the space of continuous
paths\footnote{Here we consider the space of continuous paths defined on
$[0,\infty)$ with values $\mathbb{R}^{d_{2}}$ endowed with the topology of
convergence in the supremum norm on compacts. The choice of the norm is
important. See \cite{cdfo} for further details.}, see \cite{cc} for details.
In this paper, we study the mapping $x\longrightarrow\rho_{t}^{Y\left(
\omega\right)  }(\varphi)(x)$ for a fixed (Brownian) path $Y\left(
\omega\right)  $ and a suitably chosen test function $\varphi$.$\newline$
\newline The paper is structured as follows: In Section 2 we introduce the
filtering problem and explain the connection with the randomly perturbed
semigroup (RPS). In section \ref{section main theorem statement} we state the
main results of the paper, that is, we deduce sharp gradient bounds of the
type (\ref{typeofbounds2}) for the RPS. In addition, we also give direct
corollaries on the smoothness properties of the solution of the filtering problem.

In Section \ref{main theorem section}, we derive an expansion of the RPS in
terms of a classical perturbation series. The expansion is in terms of a
series of (iterated) integrals with respect to the Brownian motion $Y$ and
derived by exploiting the intrinsic connection between the RPS and the mild
solution to the Zakai equation. We then proceed to prove the main theorem. The
proof of the main theorem is contingent on two non-trivial regularity
estimates for the terms appearing in the perturbation expansion of $\rho
_{t}^{Y\left(  \omega\right)  }$ (Propositions \ref{first a priori estimate}
and \ref{main factorial bound}), which we prove in the remainder of the paper.
$\newline$

In a first step towards proving these two propositions we re-write in Section
\ref{density expansion} the terms of the perturbation expansion iteratively
using integration by parts to derive a pathwise representation of the RPS. In
particular, this allows us to give an alternative proof of the robust
formulation of the filtering problem. In Section
\ref{section integral kernels} we then prove a priori regularity estimates for
the terms in the perturbation series. For this, we derive H\"{o}lder type
regularity estimates for each term in the pathwise representation of the
perturbation expansion by carefully leveraging the gradient estimates for heat
semi-groups due to Kusuoka and Stroock. The a priori estimates are
asymptotically sharp estimates for the lower order terms in the expansion, but
unfortunately not summable.

\bigskip

Finally, in Section \ref{factorial decay section} we rely on both the a prior
estimates derived in Section \ref{section integral kernels} and arguments
underlying the Extension Theorem - a fundamental result from rough path theory
(see, e.g. \cite{L,LCL}) - to deduce factorially decaying H\"{o}lder type
bounds for the terms in the perturbation expansion. To this end, we observe
that the terms of the original series (as derived in Section
\ref{main theorem section}), when regarded as bounded linear operators between
suitable spaces that encode the derivatives, are multiplicative functionals.
Such multiplicative functionals are more general than ordinary rough paths but
arise similarly for example also in the context of the work of Deya,
Gubinelli, Tindel et al (see e.g. \cite{DGT}) where they analyse rough heat
equations. The paper is completed with an appendix containing several useful
lemmas and an explicit description of the first three terms in the pathwise
representation or the perturbation expansion for one dimensional observations.

\emph{Acknowledgements.} The work of D. Crisan and C. Litterer was partially
supported by the EPSRC Grant No: EP/H0005500/1. The work of T. Lyons was
partially supported by the EPSRC Grant No: EP/H000100/1.

\section{The non-linear filtering problem}

Let $C_{b}^{\infty}\left(  \mathbb{R}^{N}\right)  $denote the space of smooth
bounded functions on $\mathbb{R}^{N}$ with bounded derivatives of all orders
and $C_{0}^{\infty}\left(  \mathbb{R}^{N}\right)  $ the space of compactly
supported smooth functions on $\mathbb{R}^{N}.$The nonlinear filtering problem
is stated on the probability space $(\Omega,\mathcal{F},\mathbb{{\tilde{P}})}%
$, where the new probability measure $\mathbb{{\tilde{P}}}$ is related to the
probability measure $\mathbb{P}$ under which the triple\footnote{Throughout
this section, we will omit the dependence on the initial condition
$x\in\mathbb{R}^{N}$ for the processes $X^{x}$. The same appllies to all other
processes ($Z$, $W,\mathbf{\rho}$ etc).} $(X,Y,B)$ has been introduced in the
previous section. More precisely, the probability measure $\tilde{\mathbb{P}}$
is absolutely continuous with respect to $\mathbb{P}$ and its Radon-Nikodym
derivative is given by
\[
\left.  \frac{d\tilde{\mathbb{P}}}{d\mathbb{P}}\right\vert _{\mathcal{F}_{t}%
}=Z_{t},~\ t\geq0,
\]
where $Z=\left\{  Z_{t},t\geq0\right\}  $ is the exponential martingale
defined in (\ref{zmartingale}), that is,
\[
Z_{t}=\exp\left(  \sum_{i=1}^{d_{2}}\int_{0}^{t}h_{i}\left(  X_{s}\right)
dY_{s}^{i}-\frac{1}{2}\sum_{i=1}^{d_{2}}\int_{0}^{t}h_{i}\left(  X_{s}\right)
^{2}ds\right)  ,\text{~}t\geq0.
\]
Under $\mathbb{{\tilde{P}}}$ the law of the process $X$ remains unchanged.
That is, $X$ satisfies the stochastic differential equation%
\begin{equation}
dX_{t}=V_{0}(X_{t})dt+\sum_{i=1}^{d_{1}}V_{i}(X_{t})\circ dB_{t}^{i}%
,X_{0}=x\in\mathbb{R}^{N}\quad t\geq0. \label{ss1}%
\end{equation}
As in the previous section, we assume that the vector fields $\left\{
V_{i},~\ i=0,...,d_{1}\right\}  $ are smooth and bounded with bounded
derivatives, i.e. $V_{i}\in C_{b}^{\infty}\left(  \mathbb{R}^{N}%
,\mathbb{R}^{N}\right)  ,$ and the stochastic integrals in (\ref{ss1}) are of
Stratonovich type. We denote by $\pi_{0}$ the initial distribution of $X$,
$\pi_{0}=\delta_{x}$.

Under $\mathbb{{\tilde{P}}}$ the process $Y$ is no longer a Brownian
motion,\ but becomes a semi-martingale. More precisely, $Y$ satisfies the
following evolution equation
\begin{equation}
Y_{t}=\int_{0}^{t}h(X_{s})ds+W_{t}, \label{eq:filterEq:observation}%
\end{equation}
where $W$ is a standard $\mathcal{F}_{t}$-adapted $d_{2}$-dimensional Brownian
motion (under $\mathbb{{\tilde{P}}}$) independent of $X$.\ Let $\{\mathcal{Y}%
_{t},\ t\geq0\}$ be the usual filtration associated with the process $Y$, that
is $\mathcal{Y}_{t}=\sigma(Y_{s},\ s\in\lbrack0,t])$.

Within the filtering framework, the process $X$ is called the \emph{signal}
process and the process \thinspace$Y$\ is called the \emph{observation}
process.\ The filtering problem consists in determining $\pi_{t}$,\ the
conditional distribution of the signal $X$ at time $t$ given the information
accumulated from observing $Y$ in the interval $[0,t]$, that is, for $\varphi
$\ Borel bounded function, computing
\begin{equation}
\pi_{t}\left(  \varphi\right)  =\mathbb{E}[\varphi(X_{t})\mid\mathcal{Y}_{t}].
\label{nfp}%
\end{equation}
The connection between $\pi_{t},$ the conditional distribution of the signal
$X_{t}$, and the randomly perturbed semigroup is given by the
Kallianpur-Striebel formula. We have
\begin{equation}
\pi_{t}(\varphi)=\frac{\rho_{t}^{Y\left(  \omega\right)  }(\varphi)}{\rho
_{t}^{Y\left(  \omega\right)  }(\mathbf{1})}\quad\tilde{\mathbb{P}}%
(\mathbb{P})-\mathrm{a.s.,} \label{kallstrie}%
\end{equation}
where $\mathbf{1}$\ is the constant function $\mathbf{1}\left(  x\right)  =1$
for any $x\in\mathbb{R}^{N}$.\ Equivalently, the Kallianpur-Striebel formula
can be stated as
\[
\pi_{t}=\frac{1}{c_{t}}\rho_{t}\quad\tilde{\mathbb{P}}(\mathbb{P}%
)-\mathrm{a.s.,}%
\]
where $\rho_{t}$ is the measure valued process which solves the
Duncan-Mortensen-Zakai equation (\ref{weakks})\ and $c_{t}=\rho_{t}%
(\mathbf{1})$. The Kallianpur-Striebel formula explains the usage of the term
{unnormalised} for $\rho_{t}$ as the denominator $\rho_{t}(\mathbf{1})$ can be
viewed as the normalizing factor\ for\ $\rho_{t}$. For further details of the
filtering framework see, for example, \cite{bc} and the references therein.

\section{The main theorem\label{section main theorem statement}}

In this section we will state the main results of our paper. Define the set of
all multi-indices $\mathbb{A}$ by letting
\[
\mathbb{A}=\bigcup_{k=0}^{\infty}\{0,\ldots,d_{1}\}^{k}.
\]
Following Kusuoka \cite{kusuoka} we define for multi-indices $\alpha
=(\alpha_{1},\ldots,\alpha_{k}),\beta=(\beta_{1},\ldots,\beta_{l}%
)\in\mathbb{A}$ a multiplication by setting
\[
\alpha\ast\beta=(\alpha_{1},\ldots,\alpha_{k},\beta_{1},\ldots,\beta_{l}).
\]
Furthermore we define a degree on a multi-index $\alpha$ by $\Vert\alpha
\Vert=k+card(j:\alpha_{j}=0)$. Let $A_{0}=$ $\mathbb{A}\setminus\{0\},$
$A_{1}=\mathbb{A}\setminus\{\emptyset,(0)\}$ and $A_{1}(j)=\{\alpha
\in\mathbb{A}_{1}:\Vert\alpha\Vert\leq j\}$. We inductively define a family of
vector fields indexed by $A$ by taking
\[
V_{[\emptyset]}=Id,\quad V_{[i]}=V_{i},\quad0\leq i\leq d_{1}%
\]%
\[
V_{[\alpha\ast i]}=[V_{[\alpha]},V_{i}],\quad0\leq i\leq d_{1},\alpha\in A.
\]

The following condition was introduced by Kusuoka and is weaker than the usual
(uniform) H\"{o}rmander condition imposed on the vector fields defining the
signal diffusion (see Kusuoka \cite{kusuoka} ).

\begin{definition}
The family of vector fields $V_{i}$, $i=0,\ldots,d_{1}$ is said to satisfy the
condition (UFG) if the Lie algebra generated by it is finitely generated as a
$C_{b}^{\infty}$ left module, i.e. there exists a positive $k$ such that for
all $\alpha\in A_{1}$ there exist $u_{\alpha,\beta}\in C_{b}^{\infty}\left(
\mathbb{R}^{N}\right)  $ satisfying
\begin{equation}
V_{[\alpha]}=\sum_{\beta\in A_{1}(k)}u_{\alpha,\beta}V_{[\beta]}.
\label{UFGEq}%
\end{equation}

\end{definition}

From now on suppose that our system of vector fields $V_{i}$, $i=0,\ldots
,d_{1}$ satisfies the UFG condition and let $\ell$ denote the minimal integer
$k$ for which the condition $\left(  \ref{UFGEq}\right)  $ holds. We are ready
to formulate the main theorem.

\begin{theorem}
\label{main-theorem}Suppose the family of vector fields $V_{i},$
$i=0,\ldots,d_{1}$ satisfies the UFG condition. Let $m\geq j\geq0,$
$\alpha_{1},\ldots,\alpha_{j},\ldots,\alpha_{m}\in A_{1}\left(  \ell\right)
,$ $h\in C_{b}^{\infty}\left(  \mathbb{R}^{d_{2}}\right)  .$ Then there exists
a random variable $C\left(  \omega\right)  $ almost surely finite such that
the randomly perturbed semigroup $\rho_{t}^{Y\left(  \omega\right)  }$
satisfies%
\begin{align*}
&  \left\Vert \left(  V_{\left[  \alpha_{1}\right]  }\cdots V_{\left[
\alpha_{j}\right]  }\rho_{t}^{Y\left(  \omega\right)  }\left(  V_{\left[
\alpha_{j+1}\right]  }\cdots V_{\left[  \alpha_{m}\right]  }\varphi\right)
\right)  \left(  x\right)  \right\Vert _{\infty}\\
&  \leq C\left(  \omega\right)  t^{-\left(  \left\Vert \alpha_{1}\right\Vert
+\cdots+\left\Vert \alpha_{m}\right\Vert \right)  /2}\left\Vert \varphi
\right\Vert _{\infty}%
\end{align*}
for any $\varphi\in C_{b}^{\infty}\left(  \mathbb{R}^{N}\right)  ,$
$t\in(0,1]$. If in addition $h\in C_{0}^{\infty}\left(  \mathbb{R}^{d_{2}%
}\right)  $ there exists $C\left(  \omega\right)  $ a.s. finite such that
\begin{align*}
&  \left\Vert \left(  V_{\left[  \alpha_{1}\right]  }\cdots V_{\left[
\alpha_{j}\right]  }\rho_{t}^{Y\left(  \omega\right)  }\left(  V_{\left[
\alpha_{j+1}\right]  }\cdots V_{\left[  \alpha_{m}\right]  }\varphi\right)
\right)  \left(  x\right)  \right\Vert _{p}\\
&  \leq C\left(  \omega\right)  t^{-\left(  \left\Vert \alpha_{1}\right\Vert
+\cdots+\left\Vert \alpha_{m}\right\Vert \right)  /2}\left\Vert \varphi
\right\Vert _{p}%
\end{align*}
for all $\varphi\in C_{0}^{\infty}\left(  \mathbb{R}^{N}\right)  ,$
$t\in(0,1],$ $p\in\lbrack1,\infty]$.
\end{theorem}

\begin{remark}
The random variable $C\left(  \omega\right)  $ only depends on $Y\left(
\omega\right)  $ via it's H\"{o}lder control as an Ito rough path (see Lemma
\ref{iterated integral} for details).
\end{remark}

Before we begin the proof of our main theorem we explore some immediate
consequences of the result. We first observe that we can obtain similar
estimates for the normalised conditional density.

\begin{corollary}
Under the assumptions of Theorem \ref{main-theorem} there exists a r.v.
$C\left(  \omega\right)  $ almost surely finite such that the normalised
conditional density $\pi_{t}^{{}}$ satisfies%
\begin{align*}
&  \left\Vert \left(  V_{\left[  \alpha_{1}\right]  }\cdots V_{\left[
\alpha_{j}\right]  }\pi_{t}\left(  V_{\left[  \alpha_{j+1}\right]  }\cdots
V_{\left[  \alpha_{m}\right]  }\varphi\right)  \right)  \left(  x\right)
\right\Vert _{\infty}\\
&  \leq C\left(  \omega\right)  t^{-\left(  \left\Vert \alpha_{1}\right\Vert
+\cdots+\left\Vert \alpha_{m}\right\Vert \right)  /2}\left\Vert \varphi
\right\Vert _{\infty}%
\end{align*}
for any $\varphi\in C_{b}^{\infty}\left(  \mathbb{R}^{N}\right)  ,$
$t\in(0,1]$.
\end{corollary}

\begin{proof}
We have
\begin{align}
V_{\left[  \alpha\right]  }\pi_{t}^{{}}\left(  V_{\left[  \beta\right]
}\varphi\right)  \left(  x\right)   &  =V_{\left[  \alpha\right]  }\left[
\rho_{t}^{Y\left(  \omega\right)  }\left(  V_{\left[  \beta\right]  }%
\varphi\right)  /\rho_{t}^{Y\left(  \omega\right)  }\left(  1\right)  \right]
\left(  x\right) \label{normalised}\\
&  =\frac{V_{\left[  \alpha\right]  }\rho_{t}^{Y\left(  \omega\right)
}\left(  V_{\left[  \beta\right]  }\varphi\right)  \rho_{t}^{Y\left(
\omega\right)  }\left(  1\right)  -\rho_{t}^{Y\left(  \omega\right)  }\left(
V_{\left[  \beta\right]  }\varphi\right)  V_{\left[  \alpha\right]  }\rho
_{t}^{Y\left(  \omega\right)  }\left(  1\right)  }{\left[  \rho_{t}^{Y\left(
\omega\right)  }\left(  1\right)  \right]  ^{2}}\nonumber\\
&  \leq C\left(  \omega\right)  t^{-\left(  \left\Vert \alpha\right\Vert
+\left\Vert \beta\right\Vert \right)  /2}\sup_{x\in R^{N}}\left(  1/\rho
_{t}^{Y\left(  \omega\right)  }\left(  1\right)  \right)  ^{2}\max\left(
\left\Vert \varphi\right\Vert _{\infty},1\right)  .\nonumber
\end{align}
which gives the estimates as, almost surely, (see the Appendix for a proof)%
\begin{equation}
\sup_{x\in R^{N}}\left(  1/\rho_{t}^{Y\left(  \omega\right)  }\left(
1\right)  \right)  <\infty. \label{massbound}%
\end{equation}

\end{proof}

Finally, the regularity estimates for the un-normalised conditional density
allow us to deduce estimates for the smoothness of the density of the
unnormalised conditional distribution of the signal with respect to the
Lebesgue measure. Assume that the vector fields \thinspace$V_{i}%
,i=0,...,d_{1}$ satisfy the uniform H\"{o}rmander condition and that $\pi
_{0}=\delta_{x}$ is the Dirac measure at $x$. Then
\[
\rho_{t}^{x}(\varphi)=\int_{\mathbb{R}^{d}}\varphi(y)\Psi_{t}^{x}(y)p_{t}%
^{x}\left(  y\right)  dy,
\]
where $y\rightarrow p_{t}^{x}\left(  y\right)  $ \ is the density of the law
of the signal $X_{t}^{x}$ with respect to the Lebesgue measure and
$y\rightarrow\Psi_{t}^{x}\left(  y\right)  $ is the likelihood function
\[
\Psi_{t}^{x}(y)={}\tilde{\mathbb{E}}[Z_{t}^{x}|X_{t}=y,\mathcal{Y}_{t}^{x}].
\]
We deduce from Theorem \ref{main-theorem} that
\begin{equation}
\Vert V_{\alpha}^{\ast}(\Psi_{t}^{x}p_{t}^{x})\Vert_{1}\leq Ct^{-\frac
{\Vert\alpha\Vert}{2}},\ \ \ t\in(0,1] \label{bound1}%
\end{equation}
where $V_{\alpha}^{\ast}$ is the adjoint operator of $V_{\alpha}$ for any
multi-index $\alpha\in A_{1}\left(  \ell\right)  .$

\section{\bigskip Proof of the main theorem\label{main theorem section}}

As a first step in the proof of our main theorem we expand the unnormalised
conditional distribution of the signal using its representation as the mild
solution of the Zakai equation as seen for example in \cite{ocone}. We have
\[
\rho_{t}^{Y\left(  \omega\right)  }(\varphi)\left(  x\right)  =P_{t}%
(\varphi)(x)+\sum_{i=1}^{d_{2}}\int_{0}^{t}\rho_{s}^{Y\left(  \omega\right)
}(h_{i}P_{t-s}(\varphi))\left(  x\right)  dY_{s}^{i}.
\]
To iterate this expansion we define the set of operators: $R_{\bar{t}%
,\bar{\imath}}$ where $\bar{t}=\left(  t_{1},t_{2},\ldots,t_{k}\right)  $ is a
non-empty multi-index with entries $t_{0},t_{1},\ldots,t_{k}\in\lbrack
0,\infty)\ $that have increasing values $t_{0}<t_{1}<...<t_{k}$ and
$\bar{\imath}=\left(  i_{1},...,i_{k-1}\right)  \,$\ is a multi-index with
entries $i_{1},...,i_{k-1}\in\left\{  1,2,...,d_{2}\right\}  \ $defined \
\[
R_{\left(  t_{0},t_{1}\right)  ,\varnothing}(\varphi)=P_{t_{1}-t_{0}}\left(
\varphi\right)
\]
and, inductively, for $k>1$,
\begin{align*}
R_{\left(  t_{0},t_{1},t_{2},\ldots,t_{k}\right)  },_{\left(  i_{1}%
,...,i_{k-1}\right)  }(\varphi)  &  =R_{\left(  t_{0},t_{1},\ldots
,t_{k-1}\right)  ,\left(  i_{1},...,i_{k-2}\right)  }\left(  h_{i_{k-1}%
}P_{t_{k}-t_{k-1}}(\varphi)\right) \\
&  =P_{t_{1}-t_{0}}\left(  h_{i_{1}}P_{t_{2}-t_{1}}\ldots\left(  h_{i_{k-1}%
}P_{t_{k}-t_{k-1}}(\varphi)\right)  \right) \\
&  =P_{t_{1}-t_{0}}\left(  h_{i_{1}}R_{\left(  t_{1},t_{2},\ldots
,t_{k}\right)  ,\left(  i_{2},...,i_{k-1}\right)  }(\varphi)\right)
\end{align*}
Note that the length of the multi-index $\bar{t}\ $is always two units more
than $\bar{\imath}$. In the following we will use the notation $S\left(
m\right)  $ to denote the set of all multi-indices%
\[
S\left(  m\right)  =\left\{  \left(  i_{1},...,i_{m}\right)  |\,\,1\leq
i_{j}\leq d_{2}\,,\ \ \,\,\,1\leq j\leq m\right\}  .
\]
and let $S=\bigcup_{m=1}^{\infty}S\left(  m\right)  $.

\begin{lemma}
\label{expansion lemma}We have almost surely that
\begin{equation}
\rho_{t}^{Y\left(  \omega\right)  }(\varphi)\left(  x\right)  =P_{t}%
(\varphi)(x)+\sum_{m=1}^{\infty}\sum_{\bar{\imath}\in S\left(  m\right)
}R_{0,t}^{m,\bar{\imath}}\left(  \varphi\right)  \label{expansion0}%
\end{equation}
where, for $\bar{\imath}=\left(  i_{1},...,i_{m}\right)  $,%
\[
R_{0,t}^{m,\bar{\imath}}\left(  \varphi\right)
=\underset{m\,\,\ \mathrm{times}}{\underbrace{\int_{0}^{t}\int_{0}^{t_{m}%
}\ldots\int_{0}^{t_{2}}}}R_{\left(  0,t_{1},\ldots,t_{m},t\right)
,\bar{\imath}}(\varphi)(x)dY_{t_{1}}^{i_{1}}\ldots dY_{t_{m}}^{i_{m}}.
\]

\end{lemma}

\begin{proof}
Arguing by induction it is easy to see that
\[
\rho_{t}^{Y\left(  \omega\right)  }(\varphi)\left(  x\right)  =P_{t}%
(\varphi)(x)+\sum_{m=1}^{k}\sum_{\bar{\imath}\in S\left(  m\right)  }%
R_{0,t}^{m,\bar{\imath}}\left(  \varphi\right)  +\sum_{\bar{\imath}\in
S\left(  k+1\right)  }\text{Rem}_{0,t}^{k+1,\bar{\imath}}\left(
\varphi\right)  ,
\]
where%
\[
\text{Rem}_{0,t}^{k+1,\bar{\imath}}\left(  \varphi\right)
=\underset{k+1\,\,\ \mathrm{times}}{\underbrace{\int_{0}^{t}\int_{0}^{t_{k+1}%
}\ldots\int_{0}^{t_{2}}}}\rho_{t_{1}}^{Y\left(  \omega\right)  }(h_{i_{1}%
}P_{t_{2}-t_{1}}h_{i_{2}}\cdots h_{i_{k+1}}P_{t-t_{k+1}}(\varphi))\left(
x\right)  dY_{t_{1}}^{i_{1}}\cdots dY_{t_{k+1}}^{i_{k+1}}.
\]
Using iteratively Jensen's inequality and the It\^{o} isometry we see that
\begin{align*}
&  \mathbb{E}\left[  \text{Rem}_{0,1}^{k+1,\bar{\imath}}\left(  \varphi
\right)  ^{2}\right] \\
&  \leq\int_{0}^{1}\int_{0}^{t_{k+1}}\ldots\int_{0}^{t_{2}}\mathbb{E}\left[
\rho_{t_{1}}^{Y\left(  \omega\right)  }(h_{i_{1}}P_{t_{2}-t_{1}}h_{i_{2}%
}\cdots h_{i_{k+1}}P_{t-t_{k+1}}(\varphi))^{2}\right]  dt_{1}\cdots dt_{k+1}\\
&  \leq e^{t\left\vert \left\vert h\right\vert \right\vert _{\infty}}%
\frac{\left\vert \left\vert h\right\vert \right\vert _{\infty}^{2\left(
k+1\right)  }}{\left(  k+1\right)  !}\left\Vert \varphi\right\Vert _{\infty
}^{2},
\end{align*}
since, by Jensen's inequality
\[
\mathbb{E}\left[  \rho_{t_{1}}^{Y\left(  \omega\right)  }(h_{i_{1}}%
P_{t_{2}-t_{1}}h_{i_{2}}\cdots h_{i_{k+1}}P_{t-t_{k+1}}(\varphi))^{2}\right]
\leq\left\vert \left\vert h\right\vert \right\vert _{\infty}^{2k+2}%
\mathbb{E}\left[  \left(  Z_{t}^{x}\right)  ^{2}\right]  \leq e^{t\left\vert
\left\vert h\right\vert \right\vert _{\infty}}\left\vert \left\vert
h\right\vert \right\vert _{\infty}^{2\left(  k+1\right)  }.
\]

\end{proof}

\bigskip Before we can prove the main theorem we require three non-trivial
estimates for the regularity of the terms appearing in the expansion $\left(
\ref{expansion0}\right)  $ of $\rho_{t}^{Y\left(  \omega\right)  }(\varphi).$
The first is the aforementioned gradient estimate due Kusuoka and Strook for
the heat semi-group. The following the Theorem is due to Kusuoka-Stroock
\cite{k-s} under the uniform H\"{o}rmander condition and Kusuoka
$\cite{kusuoka}$ under the UFG\ assumption.

\begin{theorem}
\label{kusuoka-stroock}Suppose the family of vector fields $V_{i},$
$i=0,\ldots,d_{1}$ satisfies the UFG condition. Let $m\geq j\geq0,$
$\alpha_{1},\ldots,\alpha_{j},\ldots,\alpha_{m}\in A_{1}\left(  \ell\right)  $
then there exists a constant $C$ such that%
\[
\left\Vert V_{\left[  \alpha_{1}\right]  }\cdots V_{\left[  \alpha_{j}\right]
}P_{t}\left(  V_{\left[  \alpha_{j+1}\right]  }\cdots V_{\left[  \alpha
_{m}\right]  }\varphi\right)  \right\Vert _{p}\leq Ct^{-\left(  \left\Vert
\alpha_{1}\right\Vert +\cdots+\left\Vert \alpha_{m}\right\Vert \right)
/2}\left\Vert \varphi\right\Vert _{p}%
\]
for any $\varphi\in C_{0}^{\infty}\left(  R^{N}\right)  ,$ $t\in(0,1]$ and
$p\in\left[  1,\infty\right]  .$
\end{theorem}

\bigskip The second ingredient for the proof of the main theorem are the
following regularity estimates for the terms $R_{0,t}^{m,\bar{\imath}}.$

\begin{proposition}
\label{first a priori estimate}Under the assumptions of Theorem
\ref{main-theorem} let $\alpha,\beta\in$ $A_{1}\left(  \ell\right)  ,$
$\gamma\in\left(  1/3,1/2\right)  $ then there exist a r.v. $C\left(
\omega,m,\gamma\right)  >0$ a.s. finite such that
\[
\left\Vert V_{\left[  \alpha\right]  }R_{0,t}^{m,\bar{\imath}}V_{\left[
\beta\right]  }\varphi\right\Vert _{\infty}\leq C\left(  \omega,m,\gamma
\right)  t^{-\left(  \left\Vert \alpha\right\Vert +\left\Vert \beta\right\Vert
\right)  /2+m\gamma}\left\Vert \varphi\right\Vert _{\infty}%
\]
for all $\bar{\imath}\in S\left(  m\right)  ~,$ $\varphi\in C_{b}^{\infty
}\left(  \mathbb{R}^{N}\right)  $ and $t\in(0,1].$
\end{proposition}

\bigskip

The preceding proposition implies that the short term asymptotics of the
regularity of $\rho_{t}$ are determined by the leading term of the expansion -
the heat semi-group $P_{t}f$ itself. The estimate is unfortunately not
summable in $m$ and will therefore only be used to control the regularity of
$R_{0,t}^{m,\bar{\imath}}$ for small $m.$ Before we proceed we state a second
set of a priori estimates that capture the regularity of the $R_{0,t}%
^{m,\bar{\imath}}$ in terms of operator norms on some carefully chosen spaces.
Note that these estimate do not lead to sharp short small time asymptotics and
will therefore only be used to estimate the regularity of $R_{0,t}%
^{m,\bar{\imath}}$ for sufficiently large values of $m. $

To derive the second set of factorially decaying estimates we regard the
$R_{0,t}^{m,\bar{\imath}}$ as linear operators acting on smooth functions
endowed with suitable norms. Noting that the heat kernels and the
multiplication operators defined by the sensor functions $h_{i}$ map
$C_{b}^{\infty}\left(  \mathbb{R}^{N}\right)  $ functions to $C_{b}^{\infty
}\left(  \mathbb{R}^{N}\right)  $ functions we see that $R_{0,t}%
^{m,\bar{\imath}}$ maps $C_{b}^{\infty}\left(  \mathbb{R}^{N}\right)  $ to
$C_{b}^{\infty}\left(  \mathbb{R}^{N}\right)  .$ We first define a
distribution space appropriate for our problem. For $\varphi\in C_{b}^{\infty
}\left(  \mathbb{R}^{N}\right)  $ let
\[
\left\Vert \varphi\right\Vert _{H^{-1}}:=\inf_{{}}\left\{  \sum_{\alpha\in
A_{0}\left(  \ell\right)  }\left\Vert \varphi_{\alpha}\right\Vert _{\infty
}:\varphi=\sum_{\alpha\in A_{0}\left(  \ell\right)  }V_{\alpha}\varphi
_{\alpha},\text{ }\varphi_{a}\in C_{b}^{\infty}\left(  \mathbb{R}^{N}\right)
\right\}  .
\]
Then $\left\Vert \cdot\right\Vert _{H^{-1}}$ defines a norm on $C_{b}^{\infty
}\left(  \mathbb{R}^{N}\right)  $ that is bounded above by $\left\Vert
\varphi\right\Vert _{\infty},$ but potentially smaller. Similarly we may
define a Sobolev type norm on $C_{b}^{\infty}\left(  \mathbb{R}^{N}\right)  $
by letting%
\[
\left\Vert \varphi\right\Vert _{H^{1}}:=\sum_{\alpha\in A_{0}\left(
\ell\right)  }\left\Vert V_{\left[  \alpha\right]  }\varphi\right\Vert
_{\infty}.
\]

Recall in this context that the index set $A_{0}\left(  \ell\right)  $
contains the empty set and we have set $V_{\left[  \emptyset\right]  }=Id$.

\bigskip

\begin{proposition}
\label{main factorial bound}Under the assumptions of Theorem
\ref{main-theorem} there exist constants $\theta>0,$ $\gamma^{\prime}%
\in\left(  1/3,1/2\right)  $ , $m_{0}\left(  \gamma^{\prime}\right)
\in\mathbb{N}$ and a random variable $c(\gamma^{\prime},\omega)$, almost
surely finite, such that
\begin{equation}
\left\Vert R_{0,t}^{m,\bar{\imath}}\right\Vert _{H^{-1}\rightarrow H^{1}}%
\leq\frac{\left(  c\left(  \gamma^{\prime},\omega\right)  t\right)
^{m\gamma^{\prime}}}{\theta\left(  m\gamma^{\prime}\right)  !}%
\end{equation}
for all $m\geq m_{0},$ $\bar{\imath}\in S\left(  m\right)  ~$and $t\in(0,1].$
\end{proposition}

Combining the previous estimates we are ready to prove our main theorem.

\begin{proof}
[Proof of Theorem \ref{main-theorem}]We are going to show that\bigskip\ there
exists a positive random variable $c\left(  \omega\right)  $ almost surely
finite such that
\[
\sup_{x_{0}\in R^{N}}\left\Vert V_{\left[  \alpha\right]  }\rho_{t}^{Y\left(
\omega\right)  }(V_{\left[  \beta\right]  }\varphi)\right\Vert _{\infty}\leq
c\left(  \omega\right)  t^{-\left(  \left\Vert \alpha\right\Vert +\left\Vert
\beta\right\Vert \right)  /2}\left\vert \left\vert \varphi\right\vert
\right\vert _{\infty}.
\]
for any $t\in(0,1]$ and $\varphi\in C_{b}^{\infty}\left(  \mathbb{R}%
^{N}\right)  $.

Fix $\gamma\in\left(  1/3,1/2\right)  $ and let $\gamma^{\prime},\theta$ and
$m_{0}$ as in Proposition \ref{main factorial bound}. We have by Lemma
\ref{expansion lemma}
\begin{align}
\left\Vert V_{\left[  \alpha\right]  }\rho_{t}^{Y\left(  \omega\right)
}(V_{\left[  \beta\right]  }\varphi)\right\Vert _{\infty}  &  \leq\left\Vert
V_{\left[  \alpha\right]  }P_{t}(V_{\left[  \beta\right]  }\varphi
)(x)\right\Vert _{\infty}\nonumber\\
&  +\sum_{k=1}^{m_{0}}\sum_{\bar{\imath}\in S\left(  k\right)  }\left\Vert
V_{\left[  \alpha\right]  }R_{0,t}^{k,\bar{\imath}}\left(  V_{\left[
\beta\right]  }\varphi\right)  \,\,\ \right\Vert _{\infty}\ \,\ +\sum
_{k=m_{0}+1}^{\infty}\sum_{\bar{\imath}\in S\left(  k\right)  }\left\Vert
V_{\left[  \alpha\right]  }R_{0,t}^{k,\bar{\imath}}\left(  V_{\left[
\beta\right]  }\varphi\right)  \,\,\ \right\Vert _{\infty}. \label{splitup}%
\end{align}
Now
\begin{align*}
\left\Vert V_{\left[  \alpha\right]  }R_{0,t}^{k,\bar{\imath}}\left(
V_{\left[  \beta\right]  }\varphi\right)  \,\,\ \right\Vert _{\infty}  &
\leq\left\Vert R_{0,t}^{k,\bar{\imath}}\left(  V_{\left[  \beta\right]
}\varphi\right)  \,\right\Vert _{H^{1}}\\
&  \leq\left\Vert R_{0,t}^{k,\bar{\imath}}\right\Vert _{H^{-1}\rightarrow
H^{1}}\left\Vert V_{\left[  \beta\right]  }\varphi\right\Vert _{H^{-1}}\\
&  \leq\left\Vert R_{0,t}^{k,\bar{\imath}}\right\Vert _{H^{-1}\rightarrow
H^{1}}\left\Vert \varphi\right\Vert _{\infty}.
\end{align*}
Therefore using Theorem $\ref{kusuoka-stroock}$ for the first, Proposition
\ref{first a priori estimate} for the second and Proposition
\ref{main factorial bound} for the third term in the sum on the right hand
side of $\left(  \ref{splitup}\right)  $ we see that
\begin{align*}
\left\Vert V_{\left[  \alpha\right]  }\rho_{t}^{Y\left(  \omega\right)
}(V_{\left[  \beta\right]  }\varphi)\right\Vert _{\infty}  &  \leq t^{-\left(
\left\Vert \alpha\right\Vert +\left\Vert \beta\right\Vert \right)
/2}\left\vert \left\vert \varphi\right\vert \right\vert _{\infty}+\sum
_{k=1}^{m_{0}}c_{k}t^{-\left(  \left\Vert \alpha\right\Vert +\left\Vert
\beta\right\Vert \right)  /2+k\gamma}\left\vert \left\vert \varphi\right\vert
\right\vert _{\infty}\\
&  +\sum_{k=m_{0}+1}^{\infty}t^{k\gamma^{\prime}}\frac{c\left(  \gamma
^{\prime},\omega,d_{2}\right)  ^{k}}{\theta\left(  k\gamma^{\prime}\right)
!}\left\vert \left\vert \varphi\right\vert \right\vert _{\infty}\,\,\ \,\\
&  \leq c\left(  \omega\right)  t^{-\left(  \left\Vert \alpha\right\Vert
+\left\Vert \beta\right\Vert \right)  /2}\left\vert \left\vert \varphi
\right\vert \right\vert _{\infty}%
\end{align*}
where%
\[
c\left(  \omega\right)  =1+\sum_{k=1}^{m_{0}}c_{k}+\sum_{k=m_{0}+1}^{\infty
}\frac{c\left(  \gamma^{\prime},\omega,d_{2}\right)  ^{k}}{\theta\left(
k\gamma^{\prime}\right)  !}.
\]
The proof may be generalised to higher derivatives by noting that all the
estimates for the smoothness of the integral kernels in section
\ref{section integral kernels} may be generalised using straightforward
induction arguments. Similarly the Sobolev and distribution spaces $H^{1}$ and
$H^{-1}$ may be generalised to accommodate higher derivatives. Clearly, the
constants and the parameter $m_{0}$ in equation $\left(  \ref{splitup}\right)
$ will depend on the number of derivatives.

For the proof of the second part of the theorem, the general \thinspace$L^{p}$
estimate, we follow Kusuoka \cite{kusuoka}. First observe that
\begin{equation}
\left\Vert \varphi\right\Vert _{1}=\sup_{%
\begin{array}
[c]{c}%
\left\Vert g\right\Vert _{\infty}\leq1\\
g\in C_{0}^{\infty}\left(  R^{N}\right)
\end{array}
}\left\vert \int\varphi g\right\vert . \label{l1 duality}%
\end{equation}
Next we identify the (formal) adjoint of the heat semi group $P_{t}\varphi.$
Let
\[
\widetilde{c}=\operatorname{div}\left(  V_{0}\right)  +\frac{1}{2}\sum
_{j=1}^{d}V_{j}\left(  \operatorname{div}\left(  V_{j}\right)  \right)
+\frac{1}{2}\sum_{j=1}^{d}\left(  \operatorname{div}\left(  V_{j}\right)
\right)  ^{2}%
\]
and
\[
\widetilde{V}_{0}=-V_{0}+\frac{1}{2}\sum_{j=1}^{d}V_{j}\left(
\operatorname{div}\left(  V_{j}\right)  \right)  .
\]
Let $\widetilde{X}_{t}$ be the diffusion associated to the vector fields
$\left(  \widetilde{V}_{0},V_{1},\ldots,V_{d}\right)  $ and define for $x\in
R^{N}$%
\[
P_{t}^{\ast}\varphi\left(  x\right)  :=E\left(  \exp\left(  \int_{0}%
^{t}\widetilde{c}\left(  \widetilde{X}_{s}^{x}\right)  ds\right)
\varphi\left(  \widetilde{X}_{t}^{x}\right)  \right)  .
\]
The following Theorem is a particular case of a result that may be found in
Kusuoka, Stroock \cite{k-s}.

\begin{theorem}
[Kusuoka-Stroock]\label{duality theorem}Let $\varphi\in C_{0}^{\infty}\left(
R^{N}\right)  $ and $g\in C_{0}^{\infty}\left(  R^{N}\right)  $ then we have%
\[
\int P_{t}\varphi\left(  x\right)  g\left(  x\right)  dx=\int\varphi\left(
x\right)  P_{t}^{\ast}g\left(  x\right)  dx,
\]
i.e. the semi group $P_{t}^{\ast}$ is the (formal) adjoint to $P_{t}.$
\end{theorem}

By Lemma \ref{expansion lemma} we may write%
\[
\rho_{t}^{Y\left(  \omega\right)  }=P_{t}+\sum_{m=1}^{\infty}\sum_{\bar
{\imath}\in S\left(  m\right)  }\int_{\Delta_{0,t}^{m}}P_{t_{1}}^{{}}H_{i_{1}%
}P_{t_{2}-t_{1}}^{{}}H_{i_{2}}\cdots H_{i_{m}}P_{t-t_{m}}^{{}}dY_{t_{1}%
}^{i_{1}}\cdots dY_{t_{m}}^{i_{m}},
\]
where $H_{i}$ are the (self-adjoint) multiplication operators corresponding
the (compactly supported) $h_{i}$. Iteratively applying Theorem
\ref{duality theorem}\ to the expansion of $\rho_{t}^{Y\left(  \omega\right)
}$ to identify the formal adjoint $\rho_{t}^{\ast}$ as
\begin{equation}
\rho_{t}^{\ast}=P_{t}^{\ast}+\sum_{m=1}^{\infty}\sum_{\bar{\imath}\in S\left(
m\right)  }\int_{\Delta_{0,t}^{m}}P_{t-t_{m}}^{\ast}H_{i_{m}}P_{t_{m}-t_{m-1}%
}^{\ast}H_{i_{m-1}}\cdots H_{i_{1}}P_{t_{1}}^{\ast}dY_{t_{1}}^{i_{1}}\cdots
dY_{t_{m}}^{i_{m}}. \label{formal adjoint}%
\end{equation}
Using $\left(  \ref{l1 duality}\right)  $ and $\left(  \ref{formal adjoint}%
\right)  $ we see that
\begin{align*}
&  \left\Vert V_{\left[  \alpha\right]  }\rho_{t}^{x}\left(  V_{\left[
\beta\right]  }\varphi\right)  \left(  Y\right)  \right\Vert _{1}\\
&  =\sup_{g\in C_{0}^{\infty},\left\Vert g\right\Vert _{\infty}\leq
1}\left\vert \int g\left(  x\right)  V_{\left[  \alpha\right]  }\rho
_{t}^{Y\left(  \omega\right)  }\left(  V_{\left[  \beta\right]  }%
\varphi\left(  x\right)  \right)  dx\right\vert \\
&  =\sup_{g\in C_{0}^{\infty},\left\Vert g\right\Vert _{\infty}\leq
1}\left\vert \int V_{\left[  \beta\right]  }^{\ast}\rho_{t}^{\ast}\left(
V_{\left[  \alpha\right]  }^{\ast}g\left(  x\right)  \right)  \varphi\left(
x\right)  dx\right\vert \\
&  \leq\sup_{g\in C_{0}^{\infty},\left\Vert g\right\Vert _{\infty}\leq
1}\left\Vert V_{\left[  \beta\right]  }^{\ast}\rho_{t}^{\ast}\left(
V_{\left[  \alpha\right]  }^{\ast}g\left(  x\right)  \right)  \right\Vert
_{\infty}\left\Vert \varphi\right\Vert _{1},
\end{align*}
where the formal adjoint of a vector field $V_{\left[  \alpha\right]  }$ is
given by
\[
V_{\left[  \alpha\right]  }^{\ast}=-V_{\left[  \alpha\right]  }-\sum_{i=1}%
^{N}\frac{\partial}{\partial x^{i}}V_{\left[  \alpha\right]  }^{i}.
\]
The arguments in the proof of Proposition \ref{first a priori estimate}
generalise easily allowing us to deduce the relevant estimates for the terms
in the expansion $\left(  \ref{formal adjoint}\right)  .$ Extending the proof
of Proposition \ref{main factorial bound} requires some small modifications
that are discussed in Remark \ref{final duality remark}. Going through the
steps in the proof of the first part of the theorem with $\rho_{t}^{\ast}$ in
place of $\rho_{t}$ we deduce that
\[
\left\Vert V_{\left[  \beta\right]  }^{\ast}\rho_{t}^{\ast}\left(  V_{\left[
\alpha\right]  }^{\ast}g\right)  \right\Vert _{\infty}\leq c\left(
\omega\right)  t^{-\left(  \left\Vert \alpha\right\Vert +\left\Vert
\beta\right\Vert \right)  /2}\left\vert \left\vert g\right\vert \right\vert
_{\infty},
\]
and the case of general $p\in\left[  1,\infty\right]  $ is a consequence of
classical Riesz-Thorin interpolation.
\end{proof}

\section{Pathwise representation of the perturbation expansion and some
preliminary estimates \label{density expansion}}

For the first step towards a proof of Proposition
\ref{first a priori estimate} we derive the pathwise representation for the
multiple stochastic integrals $R_{0,t}^{m,\bar{\imath}}\left(  \varphi\right)
$ as a sum of Riemann integrals with integrands depending on the Brownian
motion $Y$. We will require the following notation. For $k\in\mathbb{N}$ let
$\Delta_{s,t}^{k}$ denote the simplex defined by the relation
\[
s<t_{1}<\cdots<t_{k}<t
\]
and let
\[
d\overline{t}_{k}:=dt_{1}\cdots dt_{k}.
\]
For $\bar{\imath}=\left(  i_{1},...,i_{k}\right)  \in S\left(  k\right)  $ we
set%
\[
dY_{t}^{\bar{\imath}}=dY_{t_{1}}^{i_{1}}...dY_{t_{k}}^{i_{k}}.
\]

and define iterated integrals $q_{s,t}^{\bar{\imath}}\left(  Y\right)  $ by
setting
\[
q_{s,t}^{\bar{\imath}}\left(  Y\right)  :=\int_{\Delta_{s,t}^{k}}dY_{t}%
^{\bar{\imath}}=\underset{k\,\,\ \mathrm{times}}{\underbrace{\int_{s}^{t}%
\int_{s}^{t_{k}}\ldots\int_{s}^{t_{2}}}}dY_{t_{1}}^{i_{1}}...dY_{t_{k}}%
^{i_{k}}.
\]
\ Let $q_{s,\bar{t}}^{\bar{k}_{1},...,\bar{k}_{r}}\left(  Y\right)  $,
$\bar{k}_{1},...,\bar{k}_{r}\in S$, $\bar{t}=\left(  t_{1},...t_{r}\right)  $
\ be the products of iterated integrals
\[
q_{s,\bar{t}}^{\bar{k}_{1},...,\bar{k}_{r}}\left(  Y\right)  =\prod_{i=1}%
^{r}q_{s,t_{i}}^{\bar{k}_{i}}\left(  Y\right)  .
\]
Next define the sets $\Theta\left(  k\right)  $\
\[
\Theta\left(  k\right)  =\mathrm{sp}\left\{  q_{s,\bar{t}}^{\bar{k}%
_{1},...,\bar{k}_{r}}\left(  Y\right)  ,\,\ \bar{k}_{1},...,\bar{k}_{r}\in
S,\bar{t}=\left(  t_{1},...t_{r}\right)  ,\,\sum_{i=1}^{r}\left\vert \bar
{k}_{i}\right\vert =k\right\}
\]
and let $\Theta:=\bigcup_{k\in\mathbb{N}}\Theta\left(  k\right)  .$ For $q\in$
$\Theta$ we define its formal degree by setting $\deg\left(  q\right)  :=r,$
where $r$ is the unique number such that $q\in\Theta\left(  r\right)  .$ For
$\bar{\imath}=\left(  i_{1},...,i_{k}\right)  \in S\left(  k\right)  $ define
$\Phi_{\bar{\imath}},$ $\Psi_{\bar{\imath}},$ be the following operators
\begin{align*}
\Phi_{\bar{\imath}}\varphi &  =h_{i_{1}}...h_{i_{k}}\varphi\\
\Psi_{\bar{\imath}}\varphi &  =\left[  \Phi_{\bar{\imath}},A\right]  \left(
\varphi\right) \\
&  =A(h_{i_{1}}...h_{i_{k}})\varphi+\sum_{i=1}^{d}V_{i}(h_{i_{1}}...h_{i_{k}%
})V_{i}\varphi.
\end{align*}
and $\Gamma$ be the set of operators
\[
\Gamma_{\bar{\imath}}=\left\{  \Phi_{\bar{\imath}},\Psi_{\bar{\imath}}%
,\Psi_{\bar{\imath}}\Phi_{\bar{\imath}}\right\}  .
\]
In the following proposition we obtain a pathwise representation of the terms
in our expansion of the un-normalised conditional density. The proof will
exploit integration by parts formulas of the form
\[
\int_{0}^{t}q_{0,s}^{\left(  i_{1},...,i_{k}\right)  }\left(  Y\right)
\left(  \int_{0}^{s}Z_{r}dr\right)  dY_{s}^{i_{k+1}}=q_{0,t}^{\left(
i_{1},...,i_{k+1}\right)  }\left(  Y\right)  \int_{0}^{t}Z_{s}ds-\int_{0}%
^{t}q_{0,s}^{\left(  i_{1},...,i_{k+1}\right)  }\left(  Y\right)  Z_{s}ds.
\]

\begin{proposition}
\label{ibp induction}Let $\bar{\imath}=\left(  i_{1},..,i_{m}\right)  \in
S\left(  m\right)  $. Then we have, almost surely, that%
\begin{align}
R_{s,t}^{m,\bar{\imath}}\left(  \varphi\right)   &  =P_{t-s}(h_{i_{1}%
}...h_{i_{m}}\varphi)(x)q_{s,t}^{\bar{\imath}}\left(  Y\right) \nonumber\\
&  +\sum_{k=1}^{m-1}\sum_{\bar{j}_{1},...,\bar{j}_{k};\,\bar{\imath}=\bar
{j}_{1}\ast...\ast\bar{j}_{k}\,}a_{\left(  s,t\right)  }^{m,\bar{j}%
_{1},...,\bar{j}_{k}}\left(  Y\right)  \int_{\Delta_{s,t}^{k}}b_{\left(
s,t_{1},\ldots,t_{k}\right)  }^{m,\bar{j}_{1},...,\bar{j}_{k}}\left(
Y\right)  \bar{R}_{\left(  s,t_{1},\ldots,t_{k},t\right)  }^{m,\bar{j}%
_{1},...,\bar{j}_{k}}(\varphi)(x)d\overline{t}_{k}\nonumber\\
&  +\sum_{k=1}^{m}\sum_{\bar{j}_{1},...,\bar{j}_{k};\,\bar{\imath}=\bar{j}%
_{1}\ast...\ast\bar{j}_{k}\,}\int_{\Delta_{s,t}^{k}}c_{\left(  s,t_{1}%
,\ldots,t_{k}\right)  }^{m,\bar{j}_{1},...,\bar{j}_{k}}\left(  Y\right)
\hat{R}_{\left(  s,t_{1},\ldots,t_{k},t\right)  }^{m,\bar{j}_{1},...,\bar
{j}_{k}}(\varphi)(x)d\overline{t}_{k}, \label{complicatedform}%
\end{align}
and $a_{\left(  s,t\right)  }^{m,\bar{j}_{1},...,\bar{j}_{k}}\left(  Y\right)
,b_{\left(  s,t_{1},\ldots,t_{k}\right)  }^{m,\bar{j}_{1},...,\bar{j}_{k}%
}\left(  Y\right)  $, $c_{\left(  s,t_{1},\ldots,t_{k}\right)  }^{m,\bar
{j}_{1},...,\bar{j}_{k}}\left(  Y\right)  \in\Theta\;$are linear combinations
of (products of) iterated integrals of $Y\;$and $\bar{R}_{\left(
s,t_{1},\ldots,t_{k},t\right)  }^{m,\bar{j}_{1},...,\bar{j}_{k}}(\varphi)$,
respectively $\hat{R}_{\left(  s,t_{1},\ldots,t_{k},t\right)  }^{m,\bar{j}%
_{1},...,\bar{j}_{k}}(\varphi)$ are of the form%
\[
P_{t_{1}-s}\left(  \bar{\Phi}_{1}P_{t_{2}-t_{1}}\ldots\left(  \bar{\Phi}%
_{k}P_{t-t_{k}}(\varphi)\right)  \right)  ,
\]
where $\bar{\Phi}_{p}\in\Gamma_{\bar{j}_{p}},$ $p=1,..,k.$ Moreover we have
\begin{equation}
\deg\left(  a_{\left(  s,t\right)  }^{m,\bar{j}_{1},...,\bar{j}_{k}}\left(
Y\right)  \right)  +\deg\left(  b_{\left(  s,t_{1},\ldots,t_{k}\right)
}^{m,\bar{j}_{1},...,\bar{j}_{k}}\left(  Y\right)  \right)  =\deg\left(
c_{\left(  s,t_{1},\ldots,t_{k}\right)  }^{m,\bar{j}_{1},...,\bar{j}_{k}%
}\left(  Y\right)  \right)  =m. \label{degree}%
\end{equation}

\end{proposition}

\begin{proof}
The proof follows by induction. For $m=1$ observe that
\begin{align*}
R_{s,t}^{1}\left(  \varphi\right)   &  =\int_{s}^{t}P_{t_{1}-s}(h_{i_{1}%
}P_{t-t_{1}}(\varphi))(x)dY_{t_{1}}^{i_{1}}\\
&  =P_{t-s}(h_{i_{1}}\varphi)(x)\int_{s}^{t}dY_{r}^{i_{1}}-\int_{s}^{t}\left(
\int_{s}^{t_{1}}dY_{r}^{i_{1}}\right)  \frac{d}{dt_{1}}P_{t_{1}-s}(h_{i_{1}%
}P_{t-t_{1}}(\varphi))(x)dt_{1},
\end{align*}
where%
\begin{align}
\frac{d}{dt_{1}}P_{t_{1}-s}(h_{i_{1}}P_{t-t_{1}}(\varphi))(x)  &  =P_{t_{1}%
-s}(A(h_{i_{1}}P_{t-t_{1}}(\varphi)))(x)-P_{t_{1}-s}(h_{i_{1}}AP_{t-t_{1}%
}(\varphi))(x)\nonumber\\
&  =P_{t_{1}-s}(\Psi_{\left(  i_{1}\right)  }P_{t-t_{1}}(\varphi
)))(x).\nonumber
\end{align}
so (\ref{complicatedform}) holds true with%
\[
c_{\left(  s,t_{1}\right)  }^{1,\left(  i_{1}\right)  }\left(  Y\right)
=\int_{s}^{t_{1}}dY_{r}^{i_{1}}%
\]
and, obviously (\ref{degree}) is satisfied. For the induction step, observe
that for $\bar{\imath}\ast i_{m+1}$
\[
R_{s,t}^{m+1,\bar{\imath}\ast i_{m+1}}\left(  \varphi\right)  =\int_{s}%
^{t}R_{s,t_{m+1}}^{m,\bar{\imath}}(h_{i_{m+1}}P_{t-t_{m+1}}(\varphi))\left(
x\right)  dY_{t_{m+1}}^{i_{m+1}}.
\]
Hence, assuming that $R_{s,t_{m+1}}^{m,\bar{\imath}}$ has an expansion of the
(\ref{complicatedform}), \ it follows that
\begin{equation}
R_{s,t}^{m+1,\bar{\imath}\ast i_{m+1}}\left(  \varphi\right)  =R_{s,t}%
^{1,m+1,\bar{\imath}\ast i_{m+1}}\left(  \varphi\right)  +R_{s,t}%
^{2,m+1,\bar{\imath}\ast i_{m+1}}\left(  \varphi\right)  +R_{s,t}%
^{3,m+1,\bar{\imath}\ast i_{m+1}}\left(  \varphi\right)  ,
\label{inductionstep}%
\end{equation}
where%
\[
R_{s,t}^{1,m+1,\bar{\imath}\ast i_{m+1}}\left(  \varphi\right)  =\int_{s}%
^{t}P_{t_{m+1}-s}(h_{i_{1}}...h_{i_{m+1}}P_{t-t_{m+1}}(\varphi))(x)\int%
\limits_{\Delta_{s,t_{m+1}}^{m}}dY_{t}^{\bar{\imath}}dY_{t_{m+1}}^{i_{m+1}}%
\]%
\begin{multline*}
R_{s,t}^{2,m+1,\bar{\imath}\ast i_{m+1}}\left(  \varphi\right)  =\sum
_{k=1}^{m-1}\sum_{\bar{j}_{1},...,\bar{j}_{k};\,\bar{\imath}=\bar{j}_{1}%
\ast...\ast\bar{j}_{k}\,}\int_{s}^{t}a_{\left(  s,t_{m+1}\right)  }^{m,\bar
{j}_{1},...,\bar{j}_{k}}\left(  Y\right) \\
\int\limits_{\Delta_{s,t_{m+1}}^{k}}b_{\left(  s,t_{1},\ldots,t_{k}\right)
}^{m,\bar{j}_{1},...,\bar{j}_{k}}\left(  Y\right)  \bar{R}_{\left(
s,t_{1},\ldots,t_{k},t_{m+1}\right)  }^{k,\bar{j}_{1},...,\bar{j}_{k}%
}(h_{i_{m+1}}P_{t-t_{m+1}}(\varphi))(x)d\overline{t}_{k}dY_{t_{m+1}}^{i_{m+1}}%
\end{multline*}%
\begin{multline*}
R_{s,t}^{3,m+1,\bar{\imath}\ast i_{m+1}}\left(  \varphi\right)  =\\
\sum_{k=1}^{m}\sum_{\bar{j}_{1},...,\bar{j}_{k};\,\bar{\imath}=\bar{j}_{1}%
\ast...\ast\bar{j}_{k}\,}\int_{s}^{t}\int\limits_{\Delta_{s,t_{m+1}}^{k}%
}c_{\left(  s,t_{1},\ldots,t_{k}\right)  }^{m,\bar{j}_{1},...,\bar{j}_{k}%
}\left(  Y\right)  \hat{R}_{\left(  s,t_{1},\ldots,t_{k},t_{m+1}\right)
}^{k,\bar{j}_{1},...,\bar{j}_{k}}(h_{i_{m+1}}P_{t-t_{m+1}}(\varphi
))(x)d\overline{t}_{k}dY_{t_{m+1}}^{i_{m+1}}.
\end{multline*}
We expand each of the three terms in (\ref{inductionstep}). For the first term
we have%
\begin{multline}
R_{s,t}^{1,m+1,\bar{\imath}\ast i_{m+1}}\left(  \varphi\right)  =P_{t-s}%
(h_{i_{1}}...h_{i_{m+1}}\varphi)(x)\int_{\Delta_{s,t}^{m+1}}dY_{t}%
^{\bar{\imath}\ast i_{m+1}}\nonumber\\
-\int_{s}^{t}\left(  \int_{s}^{t_{m+1}}\int_{\Delta_{s,r}^{m}}dY_{t}%
^{\bar{\imath}}dY_{r}^{i_{m+1}}\right)  \frac{d}{dt_{m+1}}P_{t_{m+1}%
-s}(h_{i_{1}}...h_{i_{m+1}}P_{t-t_{m+1}}(\varphi))(x)dt_{m+1}\nonumber
\end{multline}%
\begin{multline}
=P_{t-s}(h_{i_{1}}...h_{i_{m+1}}\varphi)(x)\int_{\Delta_{s,t}^{m+1}}%
dY_{t}^{\bar{\imath}\ast i_{m+1}}\\
-\int_{s}^{t}\left(  \int_{s}^{t_{m+1}}\int_{\Delta_{s,r}^{m}}dY_{t}%
^{\bar{\imath}}dY_{r}^{i_{m+1}}\right)  P_{t_{m+1}-s}(\Psi_{\bar{\imath}\ast
i_{m+1}}P_{t-t_{m+1}}(\varphi))(x)dt_{m+1}. \label{eq5a}%
\end{multline}
so the first term in the expansion of $R_{s,t}^{1,m+1,\bar{\imath}\ast
i_{m+1}}\left(  \varphi\right)  $ gives us the first term in the expansion of
(\ref{inductionstep}) and the second term in the expansion of $R_{s,t}%
^{1,m+1,\bar{\imath}\ast i_{m+1}}\left(  \varphi\right)  $ can be incorporated
in the last term in the expansion of (\ref{inductionstep}). Obviously,%
\[
\deg\left(  \int_{\Delta_{s,t}^{m+1}}dY_{t}^{\bar{\imath}\ast i_{m+1}}\right)
=m+1
\]
so (\ref{degree}) is satisfied. For the second term we have%
\begin{multline}
R_{s,t}^{2,m+1,\bar{\imath}\ast i_{m+1}}\left(  \varphi\right)  =\sum
_{k=1}^{m-1}\sum_{\bar{j}_{1},...,\bar{j}_{k};\,\bar{\imath}=\bar{j}_{1}%
\ast...\ast\bar{j}_{k}\,}\int_{s}^{t}a_{\left(  s,t_{m+1}\right)  }^{m,\bar
{j}_{1},...,\bar{j}_{k}}\left(  Y\right) \label{eq1}\\
\int_{s}^{t_{m+1}}\int\limits_{\Delta_{s,t_{k}}^{k-1}}b_{\left(
s,t_{1},\ldots,t_{k}\right)  }^{m,\bar{j}_{1},...,\bar{j}_{k}}\left(
Y\right)  \bar{R}_{\left(  s,t_{1},\ldots,t_{k},t_{m+1}\right)  }^{m,\bar
{j}_{1},...,\bar{j}_{k}}(h_{i_{m+1}}P_{t-t_{m+1}}(\varphi))(x)d\bar{t}%
_{k-1}dY_{t_{m+1}}^{i_{m+1}}\\
=\sum_{k=1}^{m-1}\sum_{\bar{j}_{1},...,\bar{j}_{k};\,\bar{\imath}=\bar{j}%
_{1}\ast...\ast\bar{j}_{k}\,}\int_{s}^{t}a_{\left(  s,t_{m+1}\right)
}^{m,\bar{j}_{1},...,\bar{j}_{k}}\left(  Y\right)  dY_{t_{m+1}}^{i_{m+1}}%
\int_{s}^{t}S_{s,t_{m+1}}^{2,m+1,\bar{j}_{1},...,\bar{j}_{k},i_{m+1}}%
(\varphi)dt_{m+1}\\
-\sum_{k=1}^{m-1}\sum_{\bar{j}_{1},...,\bar{j}_{k};\,\bar{\imath}=\bar{j}%
_{1}\ast...\ast\bar{j}_{k}\,}\int_{s}^{t}\left(  \int_{s}^{t_{m+1}}a_{\left(
s,r\right)  }^{m,\bar{j}_{1},...,\bar{j}_{k}}\left(  Y\right)  dY_{r}%
^{i_{m+1}}\right)  S_{s,t_{m+1}}^{2,m+1,\bar{j}_{1},...,\bar{j}_{k},i_{m+1}%
}(\varphi)dt_{m+1}%
\end{multline}
where%
\begin{multline*}
S_{s,t_{m+1}}^{2,m+1,\bar{j}_{1},...,\bar{j}_{k},i_{m+1}}\left(
\varphi\right) \\
=\frac{d}{dt_{m+1}}\int\limits_{\Delta_{s,t_{m+1}}^{k}}b_{\left(
s,t_{1},\ldots,t_{k}\right)  }^{m,\bar{j}_{1},...,\bar{j}_{k}}\left(
Y\right)  \bar{R}_{\left(  s,t_{1},\ldots,t_{k},t_{m+1}\right)  }^{m,\bar
{j}_{1},...,\bar{j}_{k}}(h_{i_{m+1}}P_{t-t_{m+1}}(\varphi))(x)d\bar{t}_{k}%
\end{multline*}%
\begin{multline*}
=\int\limits_{\Delta_{s,t_{m+1}}^{k-1}}b_{\left(  s,t_{1},\ldots,t_{k}\right)
}^{m,\bar{j}_{1},...,\bar{j}_{k}}\left(  Y\right)  \bar{R}_{\left(
s,t_{1},\ldots,t_{m+1},t_{m+1}\right)  }^{m,\bar{j}_{1},...,\bar{j}_{k}%
}(h_{i_{m+1}}P_{t-t_{m+1}}(\varphi))(x)d\bar{t}_{k-1}\\
+\int_{\Delta_{s,t_{m+1}}^{k}}b_{\left(  s,t_{1},\ldots,t_{k}\right)
}^{m,\bar{j}_{1},...,\bar{j}_{k}}\left(  Y\right)  \bar{R}_{\left(
s,t_{1},\ldots,t_{k},t_{m+1}\right)  }^{m,\bar{j}_{1},...,\bar{j}_{k}%
}(A\left(  h_{i_{m+1}}P_{t-t_{m+1}}(\varphi)\right)  -h_{i_{m+1}}A\left(
P_{t-t_{m+1}}(\varphi)\right)  )(x)d\bar{t}_{k}%
\end{multline*}%
\begin{multline}
=\int\limits_{\Delta_{s,t_{m+1}}^{k-1}}b_{\left(  s,t_{1},\ldots,t_{k}\right)
}^{m,\bar{j}_{1},...,\bar{j}_{k}}\left(  Y\right)  \bar{R}_{\left(
s,t_{1},\ldots,t_{m+1},t_{m+1}\right)  }^{m,\bar{j}_{1},...,\bar{j}_{k}}%
(\Phi_{\left(  i_{m+1}\right)  }P_{t-t_{m+1}}(\varphi))(x)d\bar{t}_{k-1}\\
+\int_{\Delta_{s,t_{m+1}}^{k}}b_{\left(  s,t_{1},\ldots,t_{k}\right)
}^{m,\bar{j}_{1},...,\bar{j}_{k}}\left(  Y\right)  \bar{R}_{\left(
s,t_{1},\ldots,t_{k},t_{m+1}\right)  }^{m,\bar{j}_{1},...,\bar{j}_{k}}%
(\Psi_{\left(  i_{m+1}\right)  }P_{t-t_{m+1}}(\varphi))(x)d\bar{t}_{k}
\label{e2a}%
\end{multline}
The first term in the expansion of $R_{s,t}^{2,m+1,\bar{\imath}\ast i_{m+1}%
}\left(  \varphi\right)  $ contributes to the second term in the expansion of
(\ref{inductionstep}). The identity (\ref{degree}) is also satisfied as each
of the terms $a_{\left(  s,t_{m+1}\right)  }^{m,\bar{j}_{1},...,\bar{j}_{k}%
}\left(  Y\right)  $ is replaced by
\[
\int_{s}^{t}a_{\left(  s,t_{m+1}\right)  }^{k,m,\bar{\imath}}\left(  Y\right)
dY_{r}^{i_{m+1}}%
\]
so the degree for each term increases by 1. Similarly, the second term in the
expansion of $R_{s,t}^{2,m+1,\bar{\imath}\ast i_{m+1}}\left(  \varphi\right)
$ contributes to the third term in the expansion of (\ref{inductionstep}),
whilst the identity (\ref{degree}) is also satisfied as each of the terms
$a_{\left(  s,r\right)  }^{m,\bar{j}_{1},...,\bar{j}_{k}}\left(  Y\right)  $
is replaced by
\[
\int_{s}^{t_{m+1}}a_{\left(  s,r\right)  }^{m,\bar{j}_{1},...,\bar{j}_{k}%
}\left(  Y\right)  dY_{r}^{i_{m+1}}%
\]
so, again, the degree for each term increases by 1. Similarly,%
\begin{multline}
R_{s,t}^{3,m+1,\bar{\imath}\ast i_{m+1}}\left(  \varphi\right)  =\label{eq3}\\
\sum_{k=1}^{m}\sum_{\bar{j}_{1},...,\bar{j}_{k};\,\bar{\imath}=\bar{j}_{1}%
\ast...\ast\bar{j}_{k}\,}\int_{s}^{t}\int\limits_{\Delta_{s,t_{m+1}}^{k}%
}c_{\left(  s,t_{1},\ldots,t_{k}\right)  }^{m,\bar{j}_{1},...,\bar{j}_{k}%
}\left(  Y\right)  \hat{R}_{\left(  s,t_{1},\ldots,t_{k},t_{m+1}\right)
}^{m,\bar{j}_{1},...,\bar{j}_{k}}(h_{i_{m+1}}P_{t-t_{m+1}}(\varphi
))(x)d\bar{t}_{k}dY_{t_{m+1}}^{i_{m+1}}%
\end{multline}%
\begin{multline*}
=\sum_{k=1}^{m}\sum_{\bar{j}_{1},...,\bar{j}_{k};\,\bar{\imath}=\bar{j}%
_{1}\ast...\ast\bar{j}_{k}\,}\int_{s}^{t}dY_{t_{m+1}}^{i_{m+1}}\int_{s}%
^{t}S_{s,t_{m+1}}^{3,m,\bar{j}_{1},...,\bar{j}_{k},i_{m+1}}\left(
\varphi\right)  dt_{m+1}\\
-\int_{s}^{t}\int_{s}^{t_{m+1}}dY_{r}^{i_{m+1}}S_{s,t_{m+1}}^{3,m,\bar{j}%
_{1},...,\bar{j}_{k},i_{m+1}}\left(  \varphi\right)  dt_{m+1}%
\end{multline*}
where%
\begin{align}
S_{s,t_{m+1}}^{3,m,\bar{j}_{1},...,\bar{j}_{k},i_{m+1}}\left(  \varphi\right)
&  =\frac{d}{dt_{m+1}}\int\limits_{\Delta_{s,t_{m+1}}^{k}}c_{\left(
s,t_{1},\ldots,t_{k}\right)  }^{m,\bar{j}_{1},...,\bar{j}_{k}}\left(
Y\right)  \hat{R}_{\left(  s,t_{1},\ldots,t_{k},t_{m+1}\right)  }^{m,\bar
{j}_{1},...,\bar{j}_{k}}(h_{i_{m+1}}P_{t-t_{m+1}}(\varphi))(x)d\bar{t}%
_{k}\nonumber\\
&  =\int\limits_{\Delta_{s,t_{m+1}}^{k-1}}c_{\left(  s,t_{1},\ldots
,t_{k}\right)  }^{m,\bar{j}_{1},...,\bar{j}_{k}}\left(  Y\right)  \hat
{R}_{\left(  s,t_{1},\ldots,t_{m+1},t_{m+1}\right)  }^{m,\bar{j}_{1}%
,...,\bar{j}_{k}}(\Phi_{\left(  i_{m+1}\right)  }P_{t-t_{m+1}}(\varphi
))(x)d\bar{t}_{k-1}\nonumber\\
&  +\int\limits_{\Delta_{s,t_{m+1}}^{k}}c_{\left(  s,t_{1},\ldots
,t_{k}\right)  }^{m,\bar{j}_{1},...,\bar{j}_{k}}\left(  Y\right)  \hat
{R}_{\left(  s,t_{1},\ldots,t_{k},t_{m+1}\right)  }^{m,\bar{j}_{1},...,\bar
{j}_{k}}(\Psi_{\left(  i_{m+1}\right)  }P_{t-t_{m+1}}(\varphi))(x)d\bar{t}_{k}
\label{eq4}%
\end{align}
The first term in the expansion of $R_{s,t}^{3,m+1,\bar{\imath}\ast i_{m+1}%
}\left(  \varphi\right)  $ contributes to the second term in the expansion of
(\ref{inductionstep}). The identity (\ref{degree}) is also satisfied as we add
$\int_{s}^{t}dY_{t_{m+1}}^{i_{m+1}}$ to each of the terms so the total degree
increases by 1. Similarly, the second term in the expansion of $R_{s,t}%
^{3,m+1,\bar{\imath}\ast i_{m+1}}\left(  \varphi\right)  $ contributes to the
third term in the expansion of (\ref{inductionstep}), whilst the identity
(\ref{degree}) is again satisfied as we add $\int_{s}^{t}dY_{t_{m+1}}%
^{i_{m+1}}$ to each term.

The results now follows from (\ref{eq5a}), (\ref{eq1}), (\ref{e2a}),
(\ref{eq3}) and (\ref{eq4}).\bigskip\newline
\end{proof}

\bigskip We will require a pathwise control of the iterated (Ito) integrals
$q_{s,t}^{\bar{\imath}}\left(  Y\right)  $\ of the Brownian motion$.$ It is
well known that the Ito lift of Brownian motion is a Holder controlled rough
path (see e.g. \cite{LCL} or \cite{FV}), which immediately implies the
following lemma.

\begin{lemma}
\label{iterated integral}For any $1/3<\gamma<1/2$ there exists a positive
random variable $c=c\left(  \omega,\gamma\right)  $ and some constant
$\theta>0$ such that, almost surely,
\[
\left\vert q_{s,t}^{\bar{\imath}}\left(  Y\right)  \right\vert \leq
\frac{\left(  c\left(  \omega,\gamma\right)  \left\vert s-t\right\vert
\right)  ^{k\gamma}}{\theta\left(  k\gamma\right)  !}%
\]
for all $0\leq s\leq t\leq1,$ $\bar{\imath}\in S\left(  k\right)  .$
\end{lemma}

\bigskip

It is important to note that the operators $\Phi$ that arise when we
recursively apply the integration by parts in the Proposition
\ref{ibp induction} only involve the vector fields $V_{i}$, $i=1,\ldots,d_{1}$
(but not the vector field $V_{0}$)\ and these vector fields do not change if
we consider the Ito or Stratonovich versions of the SDE\ defining the signal.

We have already seen that the $R_{s,t}^{m,\bar{\imath}}$ may be regarded as
bounded linear operators. The next two lemmas show us how to deduce regularity
estimates on $R_{s,t}^{m,\bar{\imath}}$ from regularity estimates on the
integral kernels $\bar{R}$ and $\hat{R}.$

\begin{lemma}
\label{help 1}With the notation of Lemma \ref{ibp induction}.\ Let $\left(
W,\left\Vert \cdot\right\Vert \right)  $ be a Banach space, $\bar{\imath}\in
S\left(  m\right)  $ and suppose $R_{s,t}^{m,\bar{\imath}}\in W.$ For any
$1/3<\gamma<1/2$ there exist random variables $c(\gamma,\omega)$ such that,
almost surely
\begin{equation}
\left\Vert R_{s,t}^{m,\bar{\imath}}\right\Vert \leq\left(  c\left(
\gamma,\omega\right)  \left\vert t-s\right\vert \right)  ^{m\gamma}\sum
_{k=1}^{m}\sum_{\bar{j}_{1},...,\bar{j}_{k};\,\bar{\imath}=\bar{j}_{1}%
\ast...\ast\bar{j}_{k}\,}\int_{\Delta_{s,t}^{k}}\left\Vert \bar{R}_{\left(
s,t_{1},\ldots,t_{k},t\right)  }^{m,\bar{j}_{1},...,\bar{j}_{k}}\right\Vert
+\left\Vert \hat{R}_{\left(  s,t_{1},\ldots,t_{k},t\right)  }^{m,\bar{j}%
_{1},...,\bar{j}_{k}}\right\Vert d\overline{t}_{k}. \label{inequality}%
\end{equation}

\end{lemma}

\begin{proof}
It follows immediately from combining the H\"{o}lder estimates for the
iterated integrals $q_{s,t}^{\bar{\imath}}\left(  Y\right)  $ obtained in
Lemma \ref{iterated integral} and Proposition \ref{ibp induction} that
\begin{equation}
\left\Vert R_{s,t}^{m,\bar{\imath}}\right\Vert \leq\left(  c\left(
\gamma,\omega\right)  \left\vert t-s\right\vert \right)  ^{m\gamma}\sum
_{k=1}^{m}\sum_{\bar{j}_{1},...,\bar{j}_{k};\,\bar{\imath}=\bar{j}_{1}%
\ast...\ast\bar{j}_{k}\,}\left\Vert \int_{\Delta_{s,t}^{k}}\bar{R}_{\left(
s,t_{1},\ldots,t_{k},t\right)  }^{m,\bar{j}_{1},...,\bar{j}_{k}}d\overline
{t}_{k}\right\Vert +\left\Vert \int_{\Delta_{s,t}^{k}}\hat{R}_{\left(
s,t_{1},\ldots,t_{k},t\right)  }^{m,\bar{j}_{1},...,\bar{j}_{k}}d\overline
{t}_{k}\right\Vert .
\end{equation}

\end{proof}

In the following lemma we assume that the integral kernels $\bar{R}$ and
$\hat{R}$ have bounds with integrable singularities. The control of the
constants in the lemma is actually stronger than we will later require.

\begin{lemma}
\label{help 2}\ Under the assumptions of Lemma \ref{help 1}. Let $\bar{\imath
}\in S\left(  m\right)  ,$ $m\geq1$. Suppose there exists a constant
\thinspace$c$ such that for all $\bar{j}_{1},\ldots,$ $\bar{j}_{k}\in S$
satisfying $\bar{\imath}=\bar{j}_{1}\ast...\ast\bar{j}_{k}.,$ $t_{0}%
=0<t_{1}<\cdots<t_{k}<t$ we have both
\begin{equation}
\left\Vert \bar{R}_{\left(  0,t_{1},\ldots,t_{k},t\right)  }^{m,\bar{j}%
_{1},...,\bar{j}_{k}}\right\Vert \leq ct^{k_{0}}\frac{1}{\sqrt{t_{1}-t_{0}}%
}\cdots\frac{1}{\sqrt{t_{k}-t_{k-1}}} \label{kernel bounds}%
\end{equation}
and%
\[
\left\Vert \hat{R}_{\left(  0,t_{1},\ldots,t_{k},t\right)  }^{m,\bar{j}%
_{1},...,\bar{j}_{k}}\right\Vert \leq ct^{k_{0}}\frac{1}{\sqrt{t_{1}-t_{0}}%
}\cdots\frac{1}{\sqrt{t_{k}-t_{k-1}}}%
\]
for some $k_{0}\in\mathbb{R}$. Then \
\begin{equation}
\int_{\Delta_{s,t}^{k}}\left\Vert \bar{R}_{\left(  s,t_{1},\ldots
,t_{k},t\right)  }^{m,\bar{j}_{1},...,\bar{j}_{k}}\right\Vert +\left\Vert
\hat{R}_{\left(  s,t_{1},\ldots,t_{k},t\right)  }^{m,\bar{j}_{1},...,\bar
{j}_{k}}\right\Vert d\overline{t}_{k}\leq a_{k}\left\vert t-s\right\vert
^{k/2+k_{0}} \label{lemma4}%
\end{equation}
where%
\[
a_{k}=\frac{4\left(  2\sqrt{\pi}\right)  ^{k}c}{k\Gamma\left(  \frac{k}%
{2}\right)  }.
\]

\end{lemma}

\begin{proof}
First observe that
\[
\underset{k\,\,\ \mathrm{times}}{\underbrace{\int_{s}^{t}\int_{s}^{t_{k}%
}\ldots\int_{s}^{t_{2}}}}\bar{R}_{\left(  s,t_{1},\ldots,t_{k},t\right)
}^{m,\bar{j}_{1},...,\bar{j}_{k}}(\varphi)dt_{1}...dt_{k}%
=\underset{k\,\,\ \mathrm{times}}{\underbrace{\int_{0}^{t-s}\int_{0}^{t_{k}%
}\ldots\int_{0}^{t_{2}}}}\bar{R}_{\left(  0,t_{1},\ldots,t_{k},t-s\right)
}^{m,\bar{j}_{1},...,\bar{j}_{k}}(\varphi)dt_{1}...dt_{k}.
\]
Hence, it is sufficient to prove the result for $s=0.$ Writing $t-s=u$ let
$\Lambda_{u}$ be the simplex
\[
\Lambda_{u}=\left\{  \left(  a_{1},...,a_{k}\right)  \in\mathbb{R}_{+}%
^{k}|\sum_{i=1}^{k}a_{i}\leq u\right\}  .
\]
We have
\[
\left\Vert P_{a_{1}}\left(  \bar{\Phi}_{1}P_{a_{2}}\ldots\bar{\Phi}%
_{k-1}P_{a_{k}}\left(  \bar{\Phi}_{k}P_{u-\left(  \sum_{j=1}^{k}a_{j}\right)
}\right)  \right)  \right\Vert \leq c_{k}u^{k_{0}}\frac{1}{\sqrt{a_{1}}}%
\frac{1}{\sqrt{a_{2}}}\cdots\frac{1}{\sqrt{a_{k}}}%
\]
and introduce the change of variable $a_{i}=uz_{i}^{2},\,\,\,\,\,\ i=1,...,k$
with the determinant of its Jacobian being $2^{k}u^{k}z_{1}z_{2}...z_{k}.$
Then%
\begin{align*}
&  \int_{\Lambda_{t}^{k+1}}\left\Vert P_{a_{1}}\left(  \bar{\Phi}_{1}P_{a_{2}%
}\ldots\bar{\Phi}_{k-1}P_{a_{k}}\left(  \bar{\Phi}_{k}P_{u-\left(  \sum
_{j=1}^{k}a_{j}\right)  }\right)  \right)  \right\Vert da\\
&  \leq c_{k}2^{k}u^{\frac{k}{2}+k_{0}}l\left(  \Lambda_{1}^{k+1}\right)  ,
\end{align*}
where
\[
\Lambda_{1}^{k+1}\subset\left\{  \left(  z_{1},...,z_{k}\right)  \in
\mathbb{R}_{+}^{k}|\sum_{i=1}^{k}z_{i}^{2}\leq1\right\}
\]
In other words $\Lambda_{1}^{k+1}$ is a subset of the unit hypersphere hence
its volume less the volume of the sphere so
\[
l\left(  \Lambda_{1}^{k+1}\right)  \leq\frac{2\pi^{\frac{k}{2}}}%
{k\Gamma\left(  \frac{k}{2}\right)  }.
\]
A similar argument using $\hat{R}_{\left(  s,t_{1},\ldots,t_{k},t\right)
}^{m,\bar{j}_{1},...,\bar{j}_{k}}$ in place of $\bar{R}_{\left(
s,t_{1},\ldots,t_{k},t\right)  }^{m,\bar{j}_{1},...,\bar{j}_{k}}$ completes
the proof.
\end{proof}

\section{\bigskip Kusuoka-Stroock regularity estimates for the integral
kernels\label{section integral kernels}}

The aim of this section is to derive regularity estimates for the integral
kernels $\hat{R}$ and $\bar{R}$ that arise in the pathwise representation of
the expansion of the unnormalised conditional density. The use of these bounds
is twofold. First, they will allow us to control directly the lower order
terms in the expansion derived in section \ref{density expansion} and second
provide us via Lemma \ref{help 1} with bounds on the operator norms of the
operators $R^{m,\bar{\imath}}$ acting on the spaces $H^{1}$ and $H^{-1}$ respectively.

Recall that $\Delta_{s,t}^{k}$ denotes the simplex defined by the relation
$s<t_{1}<\cdots<t_{k}<t.$ In a first step we would like to obtain estimates
for the kernels of the form $\left(  \ref{kernel bounds}\right)  $ that are
(essentially) uniform across the simplex over which we are integrating. The
basic idea is that for any $\left(  t_{1},\ldots,t_{k}\right)  \in\Delta
_{0,t}^{k}$ there exists always at least one time interval $\left[
t_{j,}t_{j-1}\right]  $ that is of length at least $t/\left(  k+1\right)  $.
We then use the Kusuoka-Strock regularity estimates (Theorem
\ref{kusuoka-stroock}) to deduce smoothness of the heat semigroup over this
particular interval. The proof of Theorem \ref{kusuoka-stroock} employs the
methods of Malliavin calculus. As we will in the following draw on elements of
their method we recall some basic concepts of the Malliavin calculus.

\bigskip

Let $\left(  \Theta,\mathcal{H},\mu\right)  $ be the abstract Wiener space and
let $\mathcal{L}$ denote the Ornstein Uhlenbeck operator defined as in Kusuoka
\cite{kusuoka-jfa}. Denote by $G\left(  \mathcal{L}\right)  $ the set of
arbitrarily often Malliavin differentiable real valued random variables on
$\Theta$ and denote by $D_{p}^{s}$, the usual Kusuoka-Stroock Sobolev spaces
based on the Ornstein-Uhlenbeck operator (see e.g. Kusuoka \cite{kusuoka-jfa}
or \cite{kusuoka} for details ). The following definition is take from Kusuoka
\cite{kusuoka}, p.267.

\begin{definition}
Let $r\in\mathbb{R}$ and $K_{r}$ denote the set of functions $f:(0,1]\times
\mathbb{R}^{N}\rightarrow G\left(  \mathcal{L}\right)  $ satisfying the
following conditions

\begin{enumerate}
\item $f\left(  t,x\right)  $ is smooth in $x$ and $\frac{\partial^{\nu}%
f}{\partial^{\nu}x}$ is continuous in $\left(  t,x\right)  \in(0,1]\times
\mathbb{R}^{N}$ with probability one for any multi-index $\nu$

\item
\[
\sup_{t\in(0,1],x\in R^{N}}t^{-r/2}\left\Vert \frac{\partial^{\nu}f}%
{\partial^{\nu}x}\left(  t,x\right)  \right\Vert _{D_{p}^{s}}<\infty
\]
for any $s\in\mathbb{R},$ $p\in\left(  1,\infty\right)  .$
\end{enumerate}
\end{definition}

For $\Phi\in\mathcal{K}_{r}$ , $\varphi\in C_{b}^{\infty}$ define $P_{t}%
^{\Phi}\varphi=E\left(  \Phi\left(  t,x\right)  \varphi\left(  X_{t}\left(
x\right)  \right)  \right)  .$ An important ingredient in the proof of Theorem
\ref{kusuoka-stroock} which we will use repeatedly is the following Lemma
(Kusuoka \cite{kusuoka} Corollary 9).

\begin{lemma}
[Kusuoka]\label{kusuoka-lemma}Let $r\in\mathbb{R}$, $\Phi\in\mathcal{K}_{r}$
and $\alpha\in A_{1}\left(  \ell\right)  $ . Then there are $\Phi_{\alpha,1}$
, $\Phi_{\alpha,2}\in\mathcal{K}_{r-\left\Vert \alpha\right\Vert }$ such that%
\begin{equation}
P_{t}^{\Phi}V_{\left[  \alpha\right]  }=P_{t}^{\Phi_{\alpha,1}}\text{ and
}V_{\left[  \alpha\right]  }P_{t}^{\Phi}=P_{t}^{\Phi_{\alpha,2}}.
\end{equation}
Moreover there exists $C$ such that
\[
\left\Vert P_{t}^{\Phi}\varphi\right\Vert _{\infty}\leq t^{r/2}\left\Vert
\varphi\right\Vert _{\infty}%
\]
for any $\varphi\in C_{b}^{\infty}\left(  \mathbb{R}^{N}\right)  $ and
$t\in(0,1].$
\end{lemma}

Before we proceed we gather some simple properties of the spaces
$\mathcal{K}_{r}.$ The following Lemma may be found in Kusuoka \cite{kusuoka}
(Lemma 7).

\begin{lemma}
\label{kusuoka lemma 2}Let \thinspace$r_{1},r_{2}\in\mathbb{R}.$ Then

\begin{enumerate}
\item If $f_{1}\in\mathcal{K}_{r_{1`}}$ and $f_{2}\in\mathcal{K}_{r_{2}}$ then
$f_{1}f_{2}\in\mathcal{K}_{r_{1}+r_{2}}$

\item If $\varphi\in C_{b}^{\infty}\left(  \mathbb{R}^{N}\right)  $ then
$\varphi\left(  X_{t}\left(  x\right)  \right)  \in\mathcal{K}_{0}$

\item For any $\alpha,\beta$ $\in A_{1}\left(  \ell\right)  $ there exist
$a_{\alpha}^{\beta},$ $b_{\alpha}^{\beta}\in\mathcal{K}_{\left(  \left\Vert
\beta\right\Vert -\left\Vert \alpha\right\Vert \vee0\right)  }$ such that
\[
\left(  \left(  X_{t}^{-1}\right)  _{\ast}V_{\left[  \alpha\right]  }\right)
\left(  x\right)  =\sum_{\beta\in A_{1}\left(  \ell\right)  }a_{\alpha}%
^{\beta}\left(  t,x\right)  V_{\left[  \beta\right]  }\left(  x\right)
\]
and%
\[
V_{\left[  \alpha\right]  }\left(  x\right)  =\sum_{\beta\in A_{1}\left(
\ell\right)  }b_{\alpha}^{\beta}\left(  t,x\right)  \left(  \left(  X_{t}%
^{-1}\right)  _{\ast}V_{\left[  \beta\right]  }\right)  \left(  x\right)  .
\]

\end{enumerate}
\end{lemma}

\begin{proof}
The claims (2) and (3)\ are shown in \cite{kusuoka} (Lemma 7). For (1) note
that the space $\bigcap_{1<p<\infty}D_{p}^{s}\left(  \mathbb{R}\right)  $ is
an algebra (Kusuoka \cite{kusuoka-jfa} Lemma 2.13) and $\left\Vert
fg\right\Vert _{D_{p}^{s}}\leq\left\Vert f\right\Vert _{D_{r}^{s}}\left\Vert
g\right\Vert _{D_{q}^{s}}$ for $\frac{1}{p}=\frac{1}{r}+\frac{1}{q}$ , and any
$f,g\in\bigcap_{1<p<\infty}D_{p}^{s}$ and
\begin{align*}
&  \sup_{t\in(0,1],x\in R^{N}}t^{-\left(  r_{1}+r_{2}\right)  /2}\left\Vert
\frac{\partial\left(  f_{1}f_{2}\right)  }{\partial x}\left(  t,x\right)
\right\Vert _{D_{p}^{s}}\\
&  \leq\sum_{1\leq i,j\leq2,i\neq j}\sup_{t\in(0,1],x\in R^{N}}t^{-r_{i}%
/2}\left\Vert \frac{\partial f_{i}}{\partial x}\left(  t,x\right)  \right\Vert
_{D_{r}^{s}}\sup_{t\in(0,1],x\in R^{N}}t^{-r_{j}/2}\left\Vert f_{j}\left(
t,x\right)  \right\Vert _{D_{q}^{s}}<\infty
\end{align*}
The generalisation to higher derivatives is clear and the claim follows.
\end{proof}

In particular the Lemma implies that for any multi-index $\gamma,$
$p\in\lbrack1,\infty)$ and $T>0$%
\[
\sup_{x\in R^{N}}E\left[  \sup_{t\in\left[  0,T\right]  }\left\vert
\frac{\partial_{{}}^{|\gamma|}}{\partial x^{\gamma}}a_{\alpha}^{\beta}\left(
t,x\right)  \right\vert ^{p}\right]  <\infty
\]
and%
\[
\sup_{x\in R^{N}}E\left[  \sup_{t\in\left[  0,T\right]  }\left\vert
\frac{\partial^{|\gamma|}}{\partial x^{\gamma}}b_{\alpha}^{\beta}\left(
t,x\right)  \right\vert ^{p}\right]  <\infty.
\]

Let $J_{t}^{ij}\left(  x\right)  =\frac{\partial}{\partial x_{i}}X^{j}\left(
t,x\right)  $ and note that for any $C_{b}^{\infty}$ vector field $W$ we have
\[
\left(  \left(  X_{t}\right)  _{\ast}W\right)  ^{i}\left(  X_{t}\left(
x\right)  \right)  =\sum_{j=1}^{N}J_{t}^{ij}\left(  x\right)  W^{j},
\]
Suppose $\Phi\in\mathcal{K}_{r}.$ Then%
\[
V_{\left[  \alpha\right]  }P_{t}^{\Phi}\varphi\left(  x\right)  =E\left[
V_{\left[  \alpha\right]  }\Phi\varphi\left(  X_{t}\left(  x\right)  \right)
+\sum_{i,j=1}^{N}\Phi V_{\left[  \alpha\right]  }^{j}\left(  x\right)  \left(
\frac{\partial}{\partial x^{j}}\varphi\right)  \left(  X_{t}\left(  x\right)
\right)  J_{t}^{ij}\left(  x\right)  \right]
\]
It is straightforward to see that $V_{\left[  \alpha\right]  }\Phi
\in\mathcal{K}_{r}$ and for the second term in the sum we have
\begin{align*}
&  E\left[  \sum_{i,j=1}^{N}\Phi V_{\left[  \alpha\right]  }^{j}\left(
x\right)  \left(  \frac{\partial}{\partial x^{i}}\varphi\right)  \left(
X_{t}\left(  x\right)  \right)  J_{t}^{ij}\left(  x\right)  \right] \\
&  =E\left[  \sum_{i,j=1}^{N}\Phi\sum_{\beta\in A_{1}\left(  \ell\right)
}b_{\alpha}^{\beta}\left(  t,x\right)  \left(  \left(  X_{t}^{-1}\right)
_{\ast}V_{\left[  \beta\right]  }\right)  ^{j}\left(  x\right)  \left(
\frac{\partial}{\partial x^{i}}\varphi\right)  \left(  X_{t}\left(  x\right)
\right)  J_{t}^{ij}\left(  x\right)  \right] \\
&  =E\left[  \Phi\sum_{\beta\in A_{1}\left(  \ell\right)  }b_{\alpha}^{\beta
}\left(  t,x\right)  \sum_{i=1}^{N}\left(  \left(  X_{t}\right)  _{\ast
}\left(  X_{t}^{-1}\right)  _{\ast}V_{\left[  \beta\right]  }\right)
^{i}\left(  X_{t}\left(  x\right)  \right)  \left(  \frac{\partial}{\partial
x^{i}}\varphi\right)  \left(  X_{t}\left(  x\right)  \right)  \right] \\
&  =E\left[  \Phi\sum_{\beta\in A_{1}\left(  \ell\right)  }b_{\alpha}^{\beta
}\left(  t,x\right)  \sum_{i=1}^{N}V_{\left[  \beta\right]  }^{i}\left(
X_{t}\left(  x\right)  \right)  \left(  \frac{\partial}{\partial x^{i}}%
\varphi\right)  \left(  X_{t}\left(  x\right)  \right)  \right] \\
&  =\sum_{\beta\in A_{1}\left(  \ell\right)  }P^{\Phi b_{\alpha}^{\beta}%
}\left(  V_{\left[  \beta\right]  }\varphi\right)  \left(  x\right)
\end{align*}

Note that by Lemma \ref{kusuoka lemma 2} $\Phi b_{\alpha}^{\beta}\left(
t,x\right)  \in\mathcal{K}_{\left(  \left\Vert \beta\right\Vert -\left\Vert
\alpha\right\Vert \vee0\right)  +r}.$ We have just proved the following Lemma
(see e.g. Kusuoka \cite{kusuoka} Corollary 9).

\begin{lemma}
\label{forward first step}Let $\Phi\in\mathcal{K}_{r}$ and $\alpha\in
A_{1}\left(  \ell\right)  $ then $V_{\left[  \alpha\right]  }\Phi
\in\mathcal{K}_{r}$ and there exist $\Phi b_{\alpha}^{\beta}\in\mathcal{K}%
_{\left(  \left\Vert \beta\right\Vert -\left\Vert \alpha\right\Vert
\vee0\right)  +r} $ such that we have
\[
V_{\left[  \alpha\right]  }P_{t}^{\Phi}\varphi\left(  x\right)  =P^{V_{\left[
\alpha\right]  }\Phi}\varphi\left(  x\right)  +\sum_{\beta\in A_{1}\left(
\ell\right)  }P^{\Phi b_{\alpha}^{\beta}}\left(  V_{\left[  \beta\right]
}\varphi\right)  \left(  x\right)  ,
\]
for all $\varphi\in C_{b}^{\infty}\left(  \mathbb{R}^{N}\right)  .$
\end{lemma}

The following Lemma is an immediate consequence of Lemma \ref{kusuoka-lemma}.

\begin{lemma}
\label{forward-intermediate} Let $\Phi\in\mathcal{K}_{r}$ and $\alpha\in
A_{1}\left(  \ell\right)  $ then there exists $C>0$ such that
\[
\left\Vert V_{\left[  \alpha\right]  }P_{t}^{\Phi}\varphi\right\Vert _{\infty
}\leq C\sum_{\beta\in A_{0}\left(  \ell\right)  }\min\left(  t^{r/2}%
,t^{\left(  \left\Vert \beta\right\Vert -\left\Vert \alpha\right\Vert \right)
/2+r/2}\right)  \left\Vert V_{\left[  \beta\right]  }\varphi\right\Vert
_{\infty}%
\]
for all $t\in(0,1],$ $\varphi\in C_{b}^{\infty}\left(  \mathbb{R}^{N}\right)
.$ In particular, if $H$ is of the form $H=u$ $V_{i}+v$ for some $u,v$ $\in
C_{b}^{\infty},$ $i\in\left\{  1,\ldots,d_{1}\right\}  $ and $\Phi
\in\mathcal{K}_{0}$ we have%
\[
\left\Vert V_{\left[  \alpha_{1}\right]  }P_{t}^{\Phi}H\varphi\left(
x\right)  \right\Vert _{\infty}\leq C\sum_{\beta\in A_{0}\left(  \ell\right)
}\min\left(  t^{-1/2},t^{\left(  \left\Vert \beta\right\Vert -\left\Vert
\alpha\right\Vert \right)  /2-1/2}\right)  \left\Vert V_{\left[  \beta\right]
}\varphi\right\Vert _{\infty}.
\]

\end{lemma}

\begin{proof}
By Lemma \ref{forward first step} there exist $\Phi_{\beta}\in\mathcal{K}%
_{\left[  \left(  \left\Vert \beta\right\Vert -\left\Vert \alpha\right\Vert
\right)  \vee0\right]  +r.}$ such that
\begin{align*}
&  \left\Vert V_{\left[  \alpha\right]  }P_{t}^{\Phi}\varphi\left(  x\right)
\right\Vert _{\infty}\\
&  \leq\sum_{\beta\in A_{0}\left(  \ell\right)  }\left\Vert P_{t}^{\Phi
_{\beta}}V_{\left[  \beta\right]  }\varphi\right\Vert _{\infty}\\
&  \leq C\sum_{\beta\in A_{0}\left(  \ell\right)  }\min\left(  t^{r/2}%
,t^{\left(  \left\Vert \beta\right\Vert -\left\Vert \alpha\right\Vert \right)
/2+r/2}\right)  \left\Vert V_{\left[  \beta\right]  }\varphi\right\Vert
_{\infty}.
\end{align*}
The last inequality is a consequence of Lemma \ref{kusuoka-lemma} (2). To
deduce the second claim from the first of the proposition we note that by
\cite{kusuoka} Corollary 9 (2) if $\Phi\in\Phi_{a}\in\mathcal{K}_{r}$ there
exists $\Phi_{a}\in\mathcal{K}_{r-\left\Vert \alpha\right\Vert }$ such that
$P_{t}^{\Phi}V_{i}=P_{t}^{\Phi_{a}}$.
\end{proof}

Intuitively the preceding lemma provides us with a uniform (for small times)
bound when we move derivatives through the heat kernel from the outside to the inside.

We now consider the reverse situation in which we move the vector fields from
the inside to the outside. We have the following Lemma.

\begin{lemma}
\label{backward estimate}Let $\Phi\in\mathcal{K}_{r}$ and $\alpha\in
A_{1}\left(  \ell\right)  $ then there exists $\Phi_{\beta}\in\mathcal{K}_{r}$
and $\Phi a_{\alpha}^{\beta}\in\mathcal{K}_{\left(  \left\Vert \beta
\right\Vert -\left\Vert \alpha\right\Vert \vee0\right)  +r}$ such that
\[
\left(  P_{t}^{\Phi}V_{\left[  \alpha\right]  }\varphi\right)  \left(
x\right)  =\sum_{\beta\in A_{1}\left(  \ell\right)  }\left\{  V_{\left[
\beta\right]  }P_{t}^{a_{\alpha}^{\beta}\Phi}\varphi\left(  x\right)
-P_{t}^{\Phi_{\beta}}\varphi\right\}  ,
\]
for all $\varphi\in C_{b}^{\infty}\left(  \mathbb{R}^{N}\right)  .$
\end{lemma}

\begin{proof}
We have using Lemma \ref{kusuoka lemma 2} (3)%
\begin{align*}
\left(  P_{t}^{\Phi}V_{\left[  \alpha\right]  }\varphi\right)  \left(
x\right)   &  =E\left[  \Phi\sum_{i=1}^{N}V_{\left[  \alpha\right]  }%
^{i}\left(  X_{t}\left(  x\right)  \right)  \left(  \frac{\partial}{\partial
x_{i}}\varphi\right)  \left(  X_{t}\left(  x\right)  \right)  \right] \\
&  =E\left[  \Phi\sum_{i=1}^{N}\left(  \left(  X_{t}\right)  _{\ast}\left(
X_{t}^{-1}\right)  _{\ast}V_{\left[  \alpha\right]  }\right)  ^{i}\left(
X_{t}\left(  x\right)  \right)  \left(  \frac{\partial}{\partial x_{i}}%
\varphi\right)  \left(  X_{t}\left(  x\right)  \right)  \right] \\
&  =E\left[  \Phi\sum_{i,j=1}^{N}\left(  \left(  X_{t}^{-1}\right)  _{\ast
}V_{\left[  \alpha\right]  }\left(  x\right)  \right)  ^{j}J_{t}^{ij}\left(
x\right)  \left(  \frac{\partial}{\partial x_{i}}\varphi\right)  \left(
X_{t}\left(  x\right)  \right)  \right] \\
&  =E\left[  \Phi\sum_{\beta\in A_{1}\left(  \ell\right)  }a_{\alpha}^{\beta
}\left(  t,x\right)  \sum_{j=1}^{N}V_{\left[  \beta\right]  }^{j}\left(
x\right)  \sum_{i=1}^{N}J_{t}^{ij}\left(  x\right)  \left(  \frac{\partial
}{\partial x_{i}}\varphi\right)  \left(  X_{t}\left(  x\right)  \right)
\right]  .\\
&  =\sum_{\beta\in A_{1}\left(  \ell\right)  }E\left[  \Phi a_{\alpha}^{\beta
}\left(  t,x\right)  \sum_{j=1}^{N}V_{\left[  \beta\right]  }^{j}\left(
x\right)  \frac{\partial}{\partial x_{j}}\varphi\left(  X_{t}\left(  x\right)
\right)  \right]  ,
\end{align*}
where $\Phi a_{\alpha}^{\beta}\in\mathcal{K}_{\left(  \left\Vert
\beta\right\Vert -\left\Vert \alpha\right\Vert \vee0\right)  +r}.$ On the
other hand we have%
\begin{align*}
&  V_{\left[  \beta\right]  }P_{t}^{a_{\alpha}^{\beta}\Phi}\varphi\left(
x\right) \\
&  =E\left[  \Phi a_{\alpha}^{\beta}\left(  t,x\right)  \sum_{j=1}%
^{N}V_{\left[  \beta\right]  }^{j}\left(  x\right)  \frac{\partial}{\partial
x_{j}}\varphi\left(  X_{t}\left(  x\right)  \right)  \right]  +E\left[
V_{\left[  \beta\right]  }\left(  \Phi a_{\alpha}^{\beta}\right)  \left(
t,x\right)  \varphi\left(  X_{t}\left(  x\right)  \right)  \right]
\end{align*}
and deduce that
\[
\left(  P_{t}^{\Phi}V_{\left[  \alpha\right]  }\varphi\right)  \left(
x\right)  =\sum_{\beta\in A_{1}\left(  \ell\right)  }\left\{  V_{\left[
\beta\right]  }P_{t}^{a_{\alpha}^{\beta}\Phi}\varphi\left(  x\right)
-E\left[  V_{\left[  \beta\right]  }\left(  \Phi a_{\alpha}^{\beta}\right)
\left(  t,x\right)  \varphi\left(  X\left(  t,x\right)  \right)  \right]
\right\}  ,
\]
where $V_{\left[  \beta\right]  }\left(  a_{\alpha}^{\beta}\Phi\right)
\left(  t,x\right)  \in\mathcal{K}_{r}$ and $a_{\alpha}^{\beta}\in
\mathcal{K}_{\left(  \left\Vert \beta\right\Vert -\left\Vert \alpha\right\Vert
\vee0\right)  }.$
\end{proof}

The representation obtained in the previous lemma generalises to multiple heat
kernels as we observe in the following proposition.

\begin{proposition}
\label{backward prop}Let $k\in\mathbb{N},$ $\Phi_{k}\in\mathcal{K}_{r},$
$\Phi_{j}\in\mathcal{K}_{0}$ for $1\leq j<k$, $\alpha\in A_{1}\left(
\ell\right)  ,$ and $H_{j}=u_{j}V_{i_{j}}+v_{j},$ where $1\leq i_{j}\leq d,$
$u_{j},v_{j}\in C_{b}^{\infty}\left(  \mathbb{R}^{N}\right)  ,$ $j=1,\ldots
,k-1$. Then there exist $\Phi_{\beta^{1}}\in\mathcal{K}_{r_{1}},$ \ldots
,$\Phi_{\beta^{k}}\in\mathcal{K}_{r_{k}}$ such that $r_{k}\geq r$,
$r_{1},\ldots r_{k-1}\geq-1/2$ and
\begin{equation}
r_{1}+r_{2}+\cdots+r_{k}\geq\left(  \left\Vert \beta^{1}\right\Vert
-\left\Vert \alpha\right\Vert \right)  \vee0-\left(  k-1\right)  /2+r
\label{strange ineq}%
\end{equation}
and%
\begin{align*}
&  P_{t_{1}}^{\Phi_{1}}H_{1}P_{t_{2}}^{\Phi_{2}}\cdots H_{k-1}P_{t_{k}}%
^{\Phi_{k}}V_{\left[  \alpha\right]  }\varphi\left(  x\right) \\
&  =\sum_{\beta^{1}\in A_{0}\left(  \ell\right)  }\cdots\sum_{\beta^{k}\in
A_{0}\left(  \ell\right)  }V_{\left[  \beta^{1}\right]  }P_{t_{1}}%
^{\Phi_{\beta^{1}}}P_{t_{2}}^{\Phi_{\beta^{2}}}\cdots P_{t_{k}}^{\Phi
_{\beta^{k}}}\varphi\left(  x\right)
\end{align*}
holds for all $\varphi\in C_{b}^{\infty}\left(  \mathbb{R}^{N}\right)  .$
\end{proposition}

Before we begin the proof of this proposition we examine the meaning of the
assumptions on the $r_{j}.$ The assumptions $r_{1},\ldots,r_{k-1}\geq-1/2$
imply that singularities in the bounds%
\[
\left\Vert P_{t}^{\Phi_{r_{j}}}\varphi\right\Vert _{\infty}\leq t^{r_{j}%
/2}\left\Vert \varphi\right\Vert _{\infty}%
\]
in Lemma \ref{kusuoka-lemma} are integrable. The inequality $\left(
\ref{strange ineq}\right)  $ can be interpreted as follows: The left hand side
is the total regularity of the resulting expression in the proposition. For
every application of an operator $H$ we loose $1/2$ regularity reflected in
the term $-\left(  k-1\right)  /2.$ The degree of a singularity introduced by
differentiating by $V_{\left[  \alpha\right]  }$ depends on $\left\Vert
\alpha\right\Vert .$ Thus if $\left\Vert \beta\right\Vert >\left\Vert
\alpha\right\Vert $ and we replace a $V_{\left[  \alpha\right]  }$ by
$V_{\left[  \beta\right]  }$ \ we expect a compensating term, which is
captured in $\left(  \left\Vert \beta^{1}\right\Vert -\left\Vert
\alpha\right\Vert \right)  \vee0.$

\begin{proof}
As before it is by linearity sufficient to consider the case $H_{j}%
=u_{j}V_{i_{j}},$ for some $u_{j}\in C_{b}^{\infty}\left(  \mathbb{R}%
^{N}\right)  $ the case of the multiplication operator $v_{j}$ following by a
similar but easier calculation. We argue by induction, the base case being
covered by Lemma \ref{backward estimate}. For the inductive step we note that
if $\Phi_{k}\in\mathcal{K}_{0}$ then by Lemma \ref{kusuoka-lemma} there exists
$\bar{\Phi}_{k}\in\mathcal{K}_{-1/2}$ such that $P_{t}^{\Phi_{k}}uV_{i}%
=P_{t}^{\bar{\Phi}_{k}}$. Combining this fact with Lemma
\ref{backward estimate} we see
\[
P_{t_{1}}^{\Phi_{1}}H_{1}P_{t_{2}}^{\Phi_{2}}\cdots H_{k-1}P_{t_{k}}^{\Phi
_{k}}HP_{t}^{\Phi}V_{\left[  \alpha\right]  }\varphi\left(  x\right)
\]%
\begin{align*}
&  =P_{t_{1}}^{\Phi_{1}}H_{1}P_{t_{2}}^{\Phi_{2}}\cdots H_{k-1}P_{t_{k}}%
^{\Phi_{k}}uV_{i}P_{t}^{\Phi}V_{\left[  \alpha\right]  }\varphi\left(
x\right) \\
&  =P_{t_{1}}^{\Phi_{1}}H_{1}P_{t_{2}}^{\Phi_{2}}\cdots H_{k-1}P_{t_{k}%
}^{\overline{\Phi}_{k}}P_{t}^{\Phi}V_{\left[  \alpha\right]  }\varphi\left(
x\right) \\
&  =\sum_{\beta\in A_{0}\left(  \ell\right)  }P_{t_{1}}^{\Phi_{1}}%
H_{1}P_{t_{2}}^{\Phi_{2}}\cdots H_{k-1}P_{t_{k}}^{\overline{\Phi}_{k}%
}V_{\left[  \beta\right]  }P_{t}^{\Phi_{\beta}}\varphi\left(  x\right)  ,
\end{align*}
where $\Phi_{\beta}\in\mathcal{K}_{\left[  \left(  \left\Vert \beta\right\Vert
-\left\Vert \alpha\right\Vert \right)  \vee0\right]  +r}$ and $\overline{\Phi
}_{k}\in\mathcal{K}_{-1/2}$ Using the inductive hypothesis we get
\begin{align*}
&  \sum_{\beta\in A_{0}\left(  \ell\right)  }P_{t_{1}}^{\Phi_{1}}H_{1}%
P_{t_{2}}^{\Phi_{2}}\cdots H_{k-1}P_{t_{k}}^{\overline{\Phi}_{k}}V_{\left[
\beta\right]  }P_{t}^{\Phi_{\beta}}\varphi\left(  x\right) \\
&  =\sum_{\beta^{1}\in A_{0}\left(  \ell\right)  }\cdots\sum_{\beta^{k}\in
A_{0}\left(  \ell\right)  }\sum_{\beta\in A_{0}\left(  \ell\right)
}V_{\left[  \beta^{1}\right]  }P_{t_{1}}^{\Phi_{\beta^{1}}}P_{t_{2}}%
^{\Phi_{\beta^{2}}}\cdots P_{t_{k}}^{\Phi_{\beta^{k}}}P_{t}^{\Phi_{\beta}%
}\varphi\left(  x\right)  .
\end{align*}
From the inductive hypothesis we know that $\Phi_{\beta^{1}}\in\mathcal{K}%
_{r_{1}},$ \ldots,$\Phi_{\beta^{k}}\in\mathcal{K}_{r_{k}}$ such that
$r_{1},\ldots,$ $r_{k}\geq-1/2$ (using that $\overline{\Phi}_{k}\in
\mathcal{K}_{-1/2}$) and
\[
r_{1}+r_{2}+\cdots+r_{k}\geq\left(  \left\Vert \beta^{1}\right\Vert
-\left\Vert \beta\right\Vert \right)  \vee0-k/2.
\]
Hence, as required
\begin{align*}
&  \left[  \left(  \left\Vert \beta\right\Vert -\left\Vert \alpha\right\Vert
\right)  \vee0\right]  +r+r_{1}+r_{2}+\cdots+r_{k}\\
&  \geq\left[  \left(  \left\Vert \beta\right\Vert -\left\Vert \alpha
\right\Vert \right)  \vee0\right]  +r+\left(  \left\Vert \beta^{1}\right\Vert
-\left\Vert \beta\right\Vert \right)  \vee0-k/2\\
&  \geq\left(  \left\Vert \beta^{1}\right\Vert -\left\Vert \alpha\right\Vert
\right)  \vee0-k/2+r.
\end{align*}

\end{proof}

We are ready to prove the first main regularity estimate Proposition
\ref{first a priori estimate}.

\begin{proof}
[Proof of Proposition \ref{first a priori estimate}]Note that arguing as in
the proof of Lemma \ref{help 2} it is sufficient to show
\[
\left\Vert V_{\left[  \alpha\right]  }\bar{R}_{\left(  0,t_{1},\ldots
,t_{k},t\right)  }^{m,\bar{j}_{1},...,\bar{j}_{k}}V_{\left[  \beta\right]
}\varphi\right\Vert _{\infty}\leq c_{m}t^{-\left(  \left\Vert \alpha
\right\Vert +\left\Vert \beta\right\Vert \right)  /2}\frac{1}{\sqrt
{t_{1}-t_{0}}}\cdots\frac{1}{\sqrt{t_{k}-t_{k-1}}}\left\Vert \varphi
\right\Vert _{\infty}%
\]
for some constant $c_{m}$ (the bounds on $\hat{R}_{\left(  0,t_{1}%
,\ldots,t_{k},t\right)  }^{m,\bar{j}_{1},...,\bar{j}_{k}}$ follow by using the
same arguments). The functions $V_{\left[  \alpha\right]  }\bar{R}_{\left(
0,t_{1},\ldots,t_{k},t\right)  }^{m,\bar{j}_{1},...,\bar{j}_{k}}V_{\left[
\beta\right]  }\varphi$ are linear combination of terms of the form%

\[
V_{\left[  \alpha\right]  }P_{t_{1}}^{\Phi}H_{1}P_{t_{2}-t_{1}}^{\Phi}\cdots
P_{t_{k}-t_{k-1}}^{\Phi}H_{k}P_{t-t_{k}}^{\Phi}V_{\left[  \beta\right]
}\varphi
\]
for some $\Phi\in\mathcal{K}_{0}$ and $H_{j}=u_{j}$ $V_{i_{j}}+v_{j}$ with
$u_{j},v_{j}$ $\in C_{b}^{\infty}$ $.$ Recall the convention $t=t_{k+1}.$
Suppose $\left[  t_{j-1},t_{j}\right]  $ is the maximal subinterval, i.e.
satisfies
\begin{equation}
t_{j}-t_{j-1}=\max_{i=1,\ldots k+1}\left(  t_{i}-t_{i+1}\right)  \geq\frac
{t}{k} \label{largest step}%
\end{equation}
For notational reasons we have to treat the case $j=k+1$ separately, however
it will be clear from the proof that the same arguments apply in this case.

Suppose now that $j\in\{1,\ldots,k\}$ , then by Proposition
\ref{backward prop} we observe that%
\begin{align*}
&  P_{t_{j+1}-t_{j}}^{\Phi}H_{j+1}P_{t_{j+2}-t_{j+1}}^{\Phi}\cdots
H_{k}P_{t-t_{k}}^{\Phi}V_{\left[  \beta\right]  }\varphi\left(  x\right) \\
&  =\sum_{\beta^{j+1}\in A_{0}\left(  \ell\right)  }\cdots\sum_{\beta^{k}\in
A_{0}\left(  \ell\right)  }G_{\beta^{j+1},\ldots,\beta^{k}}\left(  x\right)  ,
\end{align*}
where
\[
G_{\beta^{j+1},\ldots,\beta^{k}}:=V_{\left[  \beta^{j+1}\right]  }%
P_{t_{j+1}-t_{j}}^{\Phi_{\beta^{j+1}}}P_{t_{j+2}-t_{j+1}}^{\Phi_{\beta^{j+2}}%
}\cdots P_{t-t_{k}}^{\Phi_{\beta^{k+1}}}\varphi,
\]
for some functionals $\Phi_{\beta^{j+1}}\in\mathcal{K}_{r_{1}},$ \ldots
,$\Phi_{\beta^{k+1}}\in\mathcal{K}_{r_{k}}$ with $r_{k+1}\geq0$,
$r_{j+1},\ldots r_{k}\geq-1/2$ and $r_{j+1}+r_{2}+\cdots+r_{k+1}\geq\left(
\left\Vert \beta^{j+1}\right\Vert -\left\Vert \beta\right\Vert \right)
\vee0-\left(  k-j\right)  /2.$

It follows from the maximality of $\left[  t_{j},t_{j-1}\right]  $ that%
\begin{align*}
&  \left(  t_{j+1}-t_{j}\right)  ^{r_{j+1}}\cdots\left(  t_{k}-t_{k-1}\right)
^{r_{k}}\\
&  \leq\left(  t_{j+1}-t_{j}\right)  ^{-1/2}\cdots\left(  t_{k}-t_{k-1}%
\right)  ^{-1/2}\left(  t_{j}-t_{j-1}\right)  ^{\left[  \left(  \left\Vert
\beta^{j+1}\right\Vert -\left\Vert \beta\right\Vert \right)  \vee0\right]
/2}.
\end{align*}
\newline On the other hand, to pass the derivative $V_{\left[  \alpha\right]
}$ to $P_{t_{j}-t_{j-1}}^{\Phi}$ we will iteratively use Lemma
\ref{forward-intermediate}. Once again by maximality of $\left[  t_{j}%
,t_{j-1}\right]  $ it follows that
\begin{align*}
&  \left(  t_{1}-t_{0}\right)  ^{-1/2\vee\left(  \left\Vert \beta
^{1}\right\Vert -\left\Vert \alpha\right\Vert \right)  /2}\cdots\left(
t_{j-1}-t_{j-2}\right)  ^{-1/2\vee\left(  \left\Vert \beta^{j-1}\right\Vert
-\left\Vert \beta^{j-2}\right\Vert \right)  /2}\\
&  \leq\left(  t_{1}-t_{0}\right)  ^{-1/2}\cdots\left(  t_{j-1}-t_{j-2}%
\right)  ^{-1/2}\left(  t_{j}-t_{j-1}\right)  ^{\left[  \left(  \left\Vert
\beta^{j-1}\right\Vert -\left\Vert \alpha\right\Vert \right)  \vee0\right]
/2}.
\end{align*}
Using Lemma \ref{forward-intermediate} iteratively we see from our preceding
observations that
\begin{align*}
&  \left\Vert V_{\left[  \alpha\right]  }P_{t_{1}}^{\Phi}H_{1}P_{t_{2}-t_{1}%
}^{\Phi}\cdots P_{t_{k}-t_{k-1}}^{\Phi}H_{k}P_{t-t_{k}}^{\Phi}V_{\left[
\beta\right]  }\varphi\right\Vert _{\infty}\\
&  =\left\Vert \sum_{\beta^{j+1}\in A_{0}\left(  \ell\right)  }\cdots
\sum_{\beta^{k}\in A_{0}\left(  \ell\right)  }V_{\left[  \alpha\right]
}P_{t_{1}}^{\Phi}H_{1}\cdots H_{j-1}P_{t_{j}-t_{j-1}}^{\Phi}H_{j}%
G_{\beta^{j+1},\ldots,\beta^{k}}\right\Vert _{\infty}\\
&  \leq\widetilde{C}^{j}\frac{1}{\sqrt{t_{1}-t_{0}}}\cdots\frac{1}%
{\sqrt{t_{j-1}-t_{j-2}}}\\
&  \sum_{\beta^{j-1},\ldots,\beta^{k}\in A_{0}\left(  \ell\right)  }\left(
t_{j}-t_{j-1}\right)  ^{\left[  \left(  \left\Vert \beta^{j-1}\right\Vert
-\left\Vert \alpha\right\Vert \right)  \vee0\right]  /2}\left\Vert V_{\left[
\beta^{j-1}\right]  }P_{t_{j}-t_{j-1}}^{\Phi}H_{j}G_{\beta^{j+1},\ldots
,\beta^{k}}\right\Vert _{\infty}\\
&  \leq\widetilde{C}^{k}\frac{1}{\sqrt{t_{1}-t_{0}}}\cdots\frac{1}{\sqrt
{t_{k}-t_{k-1}}}\left\Vert \varphi\right\Vert _{\infty}\\
&  \sum_{\beta^{j-1},\beta^{j+1}\in A_{0}\left(  \ell\right)  }\left(
t_{j}-t_{j-1}\right)  ^{\left[  \left(  \left\Vert \beta^{j-1}\right\Vert
-\left\Vert \alpha\right\Vert \right)  \vee0\right]  /2+\left[  \left(
\left\Vert \beta^{j+1}\right\Vert -\left\Vert \beta\right\Vert \right)
\vee0\right]  /2-\left(  \left\Vert \beta^{j-1}\right\Vert +\left\Vert
\beta^{j+1}\right\Vert \right)  /2}\\
&  \leq C^{k}\frac{1}{\sqrt{t_{1}-t_{0}}}\cdots\frac{1}{\sqrt{t_{k}-t_{k-1}}%
}t^{-\left(  \left\Vert \alpha\right\Vert +\left\Vert \beta\right\Vert
\right)  /2}\left\Vert \varphi\right\Vert _{\infty},
\end{align*}
where the penultimate inequality used Lemma \ref{kusuoka-lemma}.
\end{proof}

\section{Proof of Proposition \ref{main factorial bound}: Factorial decay of
the integral summands via rough path techniques
\label{factorial decay section}}

\subsection{Some preliminary estimates}

Before we can proceed with the proof of Proposition \ref{main factorial bound}
we explore some of the consequences of the estimates derived in the proof of
Proposition \ref{first a priori estimate}.

\begin{lemma}
\label{a priori estimate}With the notation of Lemma \ref{expansion lemma} for
any $0<\gamma<1/2,$ $m>0$ there exist random variables $c(\gamma,m,\omega
)\;$such that, almost surely
\begin{equation}
\left\Vert R_{s,t}^{m,\bar{\imath}}\right\Vert _{H^{-1}\rightarrow H^{-1}}\leq
c\left(  \gamma,m,\omega\right)  \left\vert t-s\right\vert ^{m\gamma}.
\label{ap1}%
\end{equation}%
\begin{equation}
\left\Vert R_{s,t}^{m,\bar{\imath}}\right\Vert _{H^{1}\rightarrow H^{1}}\leq
c\left(  \gamma,m,\omega\right)  \left\vert t-s\right\vert ^{m\gamma}.
\label{ap2}%
\end{equation}
and finally%
\begin{equation}
\left\Vert R_{s,t}^{m,\bar{\imath}}\right\Vert _{H^{-1}\rightarrow H^{1}}\leq
c\left(  \gamma,m,\omega\right)  \left\vert t-s\right\vert ^{m\gamma-2\ell}.
\label{ap3}%
\end{equation}
for all $\bar{\imath}\in S\left(  m\right)  ,$ $0<s<t<1$.
\end{lemma}

\begin{proof}
For all $\bar{j}_{1},\ldots,$ $\bar{j}_{k}\in S$ such that $\bar{\imath}%
=\bar{j}_{1}\ast...\ast\bar{j}_{k}$ we note that for any $0<t\leq1$ we have by
iteratively applying Lemma \ref{forward-intermediate}%

\begin{align*}
\left\Vert \bar{R}_{\left(  t_{0},t_{1},\ldots,t_{k},t\right)  }^{m,\bar
{j}_{1},...,\bar{j}_{k}}\varphi\right\Vert _{H^{1}}  &  =\sum_{\alpha\in
A_{0}\left(  \ell\right)  }\left\Vert V_{\left[  \alpha\right]  }\left(
\bar{R}_{\left(  t_{0},t_{1},\ldots,t_{k},t\right)  }^{m,\bar{j}_{1}%
,...,\bar{j}_{k}}\varphi\right)  \right\Vert _{\infty}\\
&  \leq C^{k}\frac{1}{\sqrt{t_{1}-t_{0}}}\cdots\frac{1}{\sqrt{t_{k}-t_{k-1}}%
}\sum_{\beta\in A_{0}\left(  \ell\right)  }\left\Vert V_{\left[  \beta\right]
}\varphi\right\Vert _{\infty}.
\end{align*}
The bound on $\left\Vert R_{s,t}^{m,\bar{\imath}}\left(  \varphi\right)
\right\Vert _{H^{1}\rightarrow H^{1}}$ now follows by applying Lemmas
\ref{help 1} and \ref{help 2}$.$ Finally to show inequalities $\left(
\ref{ap2}\right)  $ and $\left(  \ref{ap3}\right)  $ we let $\varphi\in
H^{-1}.$ Then there exist for every $\varepsilon>0$ functions $\varphi^{\beta
}$ such that%
\[
\varphi=\sum_{\beta\in A_{0}\left(  \ell\right)  }V_{\left[  \beta\right]
}\varphi^{\beta}%
\]
and $\sum_{\beta\in A_{0}\left(  \ell\right)  }\left\Vert \varphi^{\beta
}\right\Vert _{\infty}\leq\left\Vert \varphi\right\Vert _{H^{-1}}%
+\varepsilon.$ First we have%
\[
\left\Vert \bar{R}_{\left(  t_{0},t_{1},\ldots,t_{k},t\right)  }^{m,\bar
{j}_{1},...,\bar{j}_{k}}\varphi\right\Vert _{H^{-1}}=\left\Vert \sum_{\beta\in
A_{0}\left(  \ell\right)  }\bar{R}_{\left(  t_{0},t_{1},\ldots,t_{k},t\right)
}^{m,\bar{j}_{1},...,\bar{j}_{k}}V_{\left[  \beta\right]  }\varphi^{\beta
}\right\Vert _{H^{-1}}%
\]
and by Proposition \ref{backward prop} for each $\beta\in A_{0}\left(
\ell\right)  $ there exist functionals $\Phi_{\beta^{1}}\in\mathcal{K}_{r_{1}%
}, $ \ldots,$\Phi_{\beta^{k}}\in\mathcal{K}_{r_{k}}$ such that $r_{k}\geq0$,
$r_{1},\ldots r_{k-1}\geq-1/2$ and%
\[
\bar{R}_{\left(  t_{0},t_{1},\ldots,t_{k},t\right)  }^{m,\bar{j}_{1}%
,...,\bar{j}_{k}}V_{\left[  \beta\right]  }\varphi^{\beta}=\sum_{\beta\in
A_{0}\left(  \ell\right)  }\sum_{\beta^{1}\in A_{0}\left(  \ell\right)
}\cdots\sum_{\beta^{k}\in A_{0}\left(  \ell\right)  }V_{\left[  \beta
^{1}\right]  }P_{t_{1}-t_{0}}^{\Phi_{\beta^{1}}}P_{t_{2}-t_{1}}^{\Phi
_{\beta^{2}}}\cdots P_{t-t_{k}}^{\Phi_{\beta^{k}}}\varphi.
\]
We deduce from Lemma \ref{kusuoka-lemma} that
\begin{align*}
\left\Vert \bar{R}_{\left(  t_{0},t_{1},\ldots,t_{k},t\right)  }^{m,\bar
{j}_{1},...,\bar{j}_{k}}V_{\left[  \beta\right]  }\varphi^{\beta}\right\Vert
_{H^{-1}}  &  \leq\sum_{\beta\in A_{0}\left(  \ell\right)  }\sum_{\beta^{1}\in
A_{0}\left(  \ell\right)  }\cdots\sum_{\beta^{k}\in A_{0}\left(  \ell\right)
}\left\Vert P_{t_{1}-t_{0}}^{\Phi_{\beta^{1}}}P_{t_{2}-t_{1}}^{\Phi_{\beta
^{2}}}\cdots P_{t-t_{k}}^{\Phi_{\beta^{k}}}\varphi^{\beta}\right\Vert
_{\infty}\\
&  \leq C_{k}\frac{1}{\sqrt{t_{1}-t_{0}}}\cdots\frac{1}{\sqrt{t_{k}-t_{k-1}}%
}\left\Vert \varphi^{\beta}\right\Vert _{\infty}%
\end{align*}
and consequently
\begin{align*}
\left\Vert \bar{R}_{\left(  t_{0},t_{1},\ldots,t_{k},t\right)  }^{m,\bar
{j}_{1},...,\bar{j}_{k}}\varphi\right\Vert _{H^{-1}}  &  \leq c_{k}\frac
{1}{\sqrt{t_{1}-t_{0}}}\cdots\frac{1}{\sqrt{t_{k}-t_{k-1}}}\sum_{\beta\in
A_{0}\left(  \ell\right)  }\left\Vert \varphi^{\beta}\right\Vert _{\infty}\\
&  \leq c_{k}\frac{1}{\sqrt{t_{1}-t_{0}}}\cdots\frac{1}{\sqrt{t_{k}-t_{k-1}}%
}\left\Vert \varphi\right\Vert _{H^{-1}}+\varepsilon.
\end{align*}
To demonstrate the last inequality observe that arguing exactly as in the
proof of Proposition \ref{first a priori estimate} we have,
\begin{align*}
\left\Vert \bar{R}_{\left(  t_{0},t_{1},\ldots,t_{k},t\right)  }^{m,\bar
{j}_{1},...,\bar{j}_{k}}\varphi\right\Vert _{H^{1}}  &  =\sum_{\alpha\in
A_{0}\left(  \ell\right)  }\left\Vert V_{\left[  \alpha\right]  }\left(
\bar{R}_{\left(  t_{0},t_{1},\ldots,t_{k},t\right)  }^{m,\bar{j}_{1}%
,...,\bar{j}_{k}}\varphi\right)  \right\Vert _{\infty}\\
&  \leq\sum_{\alpha\in A_{0}\left(  \ell\right)  }\sum_{\beta\in A_{0}\left(
\ell\right)  }\left\Vert V_{\left[  \alpha\right]  }\left(  \bar{R}_{\left(
t_{0},t_{1},\ldots,t_{k},t\right)  }^{m,\bar{j}_{1},...,\bar{j}_{k}}V_{\left[
\beta\right]  }\varphi^{\beta}\right)  \right\Vert _{\infty}%
\end{align*}%
\begin{align*}
&  \leq c_{k}\frac{1}{\sqrt{t_{1}-t_{0}}}\cdots\frac{1}{\sqrt{t_{k}-t_{k-1}}%
}\sum_{\alpha\in A_{0}\left(  \ell\right)  }\sum_{\beta\in A_{0}\left(
\ell\right)  }t^{-\left(  \left\Vert \alpha\right\Vert +\left\Vert
\beta\right\Vert \right)  /2}\left\Vert \varphi^{\beta}\right\Vert _{\infty}\\
&  \leq c_{k}\frac{1}{\sqrt{t_{1}-t_{0}}}\cdots\frac{1}{\sqrt{t_{k}-t_{k-1}}%
}t^{-2\ell}\sum_{\beta\in A_{0}\left(  \ell\right)  }\left\Vert \varphi
^{\beta}\right\Vert _{\infty}\\
&  \leq c_{k}\frac{1}{\sqrt{t_{1}-t_{0}}}\cdots\frac{1}{\sqrt{t_{k}-t_{k-1}}%
}t^{-2\ell}\left\Vert \varphi\right\Vert _{H^{-1}}+\varepsilon,
\end{align*}
where $c_{k}$ are constants changing from line to line. The claim in both
cases now follows once again from Lemmas \ref{help 1} and \ref{help 2}$.$ . As
before we note that the same estimates apply to $\hat{R}_{\left(  t_{0}%
,t_{1},\ldots,t_{k},t\right)  }^{m,\bar{j}_{1},...,\bar{j}_{k}}$ in place of
$\bar{R}_{\left(  t_{0},t_{1},\ldots,t_{k},t\right)  }^{m,\bar{j}_{1}%
,...,\bar{j}_{k}}.$
\end{proof}

So far, we have established a priori H\"{o}lder type estimates for
$R_{s,t}^{m,\bar{\imath}}\left(  \varphi\right)  ,$ but the estimates in their
current form are not yet summable. The following proof of Proposition
\ref{main factorial bound} relies on a fundamental rough path technique to
improve on these bounds and demonstrate that the operator norms of
$R_{s,t}^{m,\bar{\imath}}$ decay in fact factorially in $m.$

\subsection{Proof of Proposition \ref{main factorial bound}}

To make the presentation more transparent we introduce some additional
notations for the following arguments. Recall that $\Delta_{s,t}^{k}$ denotes
the simplex defined by the relation $s<t_{1}<\cdots<t_{k}<t$ and the $H_{i}$
are the operators corresponding to multiplication by the sensor function
$h_{i}.$ For any $0\leq s<t\leq T$ define $R_{s,t}^{\emptyset}:=1$ and recall
the linear operators $R_{s,t}^{n,\bar{\imath}}$ may be written as%
\[
R_{s,t}^{n,\bar{\imath}}=\int_{\Delta_{s,t}^{k}}P_{t_{1}-s}H_{i_{1}}%
P_{t_{2}-t_{1}}H_{i_{2}}\cdots H_{i_{n}}P_{t-t_{n}}dY_{t_{1}}^{i_{1}}\cdots
dY_{t_{n}}^{i_{n}}.
\]
for all $\bar{\imath}=\left(  i_{1},\ldots,i_{n}\right)  \in S$, $n\geq$ $1.$

\bigskip

Let $W:=\mathbb{R}^{d_{2}}$ and $\varepsilon_{1},\ldots,\varepsilon_{d_{2}}$ a
Basis for $W.$ For $\bar{\imath}=\left(  i_{1},\ldots i_{j}\right)  \in
S\left(  j\right)  ~\ $let $\varepsilon_{\bar{\imath}}=\varepsilon_{i_{1}%
}\otimes\cdots\otimes\varepsilon_{i_{j}}$ and note that the $\varepsilon
_{\bar{\imath}}$ are a basis for the space $W^{\otimes j}.$ Finally, let $V$
be a Banach algebra (i.e. a Banach space with a multiplication and a
submultiplicative norm). We define $\mathcal{P}_{d_{2},k}\left(  V\right)  $
the space of non-commutative polynomials in $d_{2}$ variables of degree at
most $k$ over $V$ by letting%
\[
\mathcal{P}_{d_{2},k}\left(  V\right)  :=\left\{  \sum_{j=0}^{k}\sum
_{\bar{\imath}\in S\left(  j\right)  }c_{\bar{\imath}}\varepsilon_{\bar
{\imath}}:c_{\bar{\imath}}\in V\right\}  .
\]
Define a multiplication for $a=\sum_{j=0}^{k}a_{j},$ $a_{j}=\sum_{\bar{\imath
}\in S\left(  j\right)  }a_{\bar{\imath}}\varepsilon_{\bar{\imath}}$ and
$b=\sum_{j=0}^{k}b_{j},$ $b_{j}=\sum_{\bar{\imath}\in S\left(  j\right)
}b_{\bar{\imath}}\varepsilon_{\bar{\imath}}$ by setting%
\begin{equation}
ab:=\sum_{v=0}^{k}\sum_{j=0}^{v}a_{j}b_{v-j}:=\sum_{v=0}^{k}\sum_{j=0}^{v}%
\sum_{\bar{\imath}\in S\left(  j\right)  }\sum_{\bar{l}\in S\left(
v-j\right)  }a_{\bar{\imath}}b_{\bar{l}}\varepsilon_{\bar{\imath}\ast\bar{l}}.
\label{mult def}%
\end{equation}
Further note that%
\begin{equation}
\sum_{j=0}^{v}\sum_{\bar{\imath}\in S\left(  j\right)  }\sum_{\bar{l}\in
S\left(  v-j\right)  }a_{\bar{\imath}}b_{\bar{l}}\varepsilon_{\bar{\imath}%
\ast\bar{l}}=\sum_{\bar{\imath}\in S\left(  v\right)  }\sum_{\bar{m}\ast
\bar{l}=\bar{\imath}}a_{\bar{m}}b_{\bar{l}}\varepsilon_{\bar{m}\ast\bar{l}}
\label{helper eq}%
\end{equation}
and define for $k\geq i\geq1$ the projection $\pi_{i}$ by setting $\pi
_{i}\left(  a\right)  =a_{i}.$ We impose a norm on $\mathcal{P}_{d_{2}%
,k}\left(  V\right)  $ by setting
\[
\left\Vert \sum_{j=0}^{k}\sum_{\bar{\imath}\in S\left(  j\right)  }%
c_{\bar{\imath}}\varepsilon_{\bar{\imath}}\right\Vert =\sup\left\{  \left\Vert
c_{\bar{\imath}}\right\Vert :j\in\left\{  0,\ldots,k\right\}  ,\bar{\imath}\in
S\left(  j\right)  \right\}
\]
Let $Q_{s,t}^{0}=1$ and $Q_{s,t}^{j}$ for $j\in\mathbb{N}$ be given by
\[
Q_{s,t}^{j}=\sum_{\bar{\imath}\in S\left(  j\right)  }R_{s,t}^{j,\bar{\imath}%
}\varepsilon_{\bar{\imath}}.
\]
Finally, we may set
\[
Q_{s,t}^{\left[  n\right]  }=\sum_{i=0}^{n}Q_{s,t}^{i}.
\]

\bigskip

Observe that for any $s<u<t$ and $k\in\mathbb{N}$ and $\bar{\imath}=\left(
i_{1},\ldots i_{k}\right)  \in S\left(  k\right)  ,$ we have partitioning the
simplex $\Delta_{s,t}^{k}$%
\begin{align}
R_{s,t}^{k,\bar{\imath}}  &  =\int_{\Delta_{s,u}^{k}}P_{t_{1}-s}H_{i_{1}%
}P_{t_{2}-t_{1}}H_{i_{2}}\cdots H_{i_{k}}P_{t-t_{k}}dY_{t_{1}}^{i_{1}}\cdots
dY_{t_{k}}^{i_{k}}\nonumber\\
&  +\int_{\Delta_{u,t}^{k}}P_{t_{1}-s}H_{i_{1}}P_{t_{2}-t_{1}}H_{i_{2}}\cdots
H_{i_{k}}P_{t-t_{k}}dY_{t_{1}}^{i_{1}}\cdots dY_{t_{k}}^{i_{k}}\nonumber\\
&  +\sum_{j=1}^{k-1}\int_{\Delta_{s,u}^{k}}P_{t_{1}-s}H_{i_{1}}\cdots
P_{t_{j}-t_{j-1}}H_{i_{j}}P_{u-t_{j}}dY_{t_{1}}^{i_{1}}\cdots dY_{t_{j}%
}^{i_{j}}\nonumber\\
&  \int_{\Delta_{u,t}^{k-j}}P_{t_{j+1}-u}H_{i_{j+1}}P_{t_{j+2}-t_{j+1}}\cdots
H_{i_{k}}P_{t-t_{k}}dY_{t_{j+1}}^{i_{j+1}}\cdots dY_{t_{k}}^{i_{k}}\nonumber\\
&  =\sum_{j=0}^{k}R_{s,u}^{j,\left(  i_{1},\ldots i_{j}\right)  }%
R_{u,t}^{k-j,\left(  i_{j+1},\ldots i_{k}\right)  }\nonumber\\
&  =\sum_{\bar{m}\ast\bar{l}=\bar{\imath}}R_{s,u}^{\left\vert \bar
{m}\right\vert ,\bar{m}}R_{u,t}^{\left\vert \bar{l}\right\vert ,\bar{l}}
\label{mult property prep}%
\end{align}
and therefore using $\left(  \ref{helper eq}\right)  $
\begin{align*}
Q_{s,t}^{\left[  k\right]  }  &  =\sum_{v=0}^{k}\sum_{\bar{\imath}\in S\left(
v\right)  }R_{s,t}^{v,\bar{\imath}}\varepsilon_{\bar{\imath}}\\
&  =\sum_{v=0}^{k}\sum_{\bar{\imath}\in S\left(  v\right)  }\sum_{\bar{m}%
\ast\bar{l}=\bar{\imath}}R_{s,u}^{\left\vert \bar{m}\right\vert ,\bar{m}%
}R_{u,t}^{\left\vert \bar{l}\right\vert ,\bar{l}}\varepsilon_{\bar{m}%
}\varepsilon_{\bar{l}}\\
&  =\sum_{v=0}^{k}\sum_{j=0}^{v}\sum_{\bar{\imath}\in S\left(  j\right)  }%
\sum_{\bar{l}\in S\left(  v-j\right)  }R_{s,u}^{j,\bar{\imath}}R_{u,t}%
^{v-j,\bar{l}}\varepsilon_{\bar{\imath}\ast\bar{l}}\\
&  =\sum_{v=0}^{k}\sum_{j=0}^{v}Q_{s,u}^{j}Q_{u,t}^{v-j}%
\end{align*}
or equivalently
\begin{equation}
Q_{s,t}^{\left[  k\right]  }=Q_{s,u}^{\left[  k\right]  }Q_{u,t}^{\left[
k\right]  }. \label{multiplicative property}%
\end{equation}
Analogous to the corresponding rough path concept we will refer to $\left(
\ref{multiplicative property}\right)  $ as the multiplicative property. We
recall that by Lemma \ref{expansion lemma}
\begin{equation}
\rho_{t}=P_{t}+\sum_{n=1}^{\infty}\sum_{\bar{\imath}\in S\left(  n\right)
}R_{0,t}^{n,\bar{\imath}}. \label{series representation}%
\end{equation}

The following proposition demonstrates that it suffices to obtain Holder type
controls on finitely many of the $Q_{s,t}^{n}$ to control the infinite series
in $\left(  \ref{series representation}\right)  $ . The proof utilises
techniques of the classical extension theorem for rough paths due to Lyons
(see e.g. \cite{LCL} p.45f) and exploits the multiplicative structure of the
operator valued integrands.

\begin{lemma}
\label{holder control}Let $q\geq1$ and let $\lfloor q\rfloor$ denote the
integer part of $q$ and $V$ be a Banach algebra with norm $\left\Vert
\cdot\right\Vert .$ Suppose $Q^{\left[  \lfloor q\rfloor\right]  }=\sum
_{j=0}^{\lfloor q\rfloor}Q^{j}\in\mathcal{P}_{d_{2},\lfloor q\rfloor}\left(
V\right)  $ satisfies the multiplicative property $\left(
\ref{multiplicative property}\right)  $. Suppose there exists a constant $C>0$
such that for all $\left(  s,t\right)  \in\Delta_{\left[  0,1\right]  },$
$j=1,\cdots\lfloor q\rfloor,$
\begin{equation}
\left\Vert Q_{s,t}^{j}\right\Vert \leq\frac{\left(  C\left\vert t-s\right\vert
\right)  ^{j/q}}{\theta\left(  j/q\right)  !}, \label{holder bound}%
\end{equation}
where $\theta=\left(  q^{2}+\sum_{r=3}^{\infty}\left(  \frac{2}{r-2}\right)
^{\frac{\lfloor q\rfloor+1}{q}}\right)  .$Then for all $m>\lfloor q\rfloor$
there exists a multiplicative extension $1+$ $Q_{s,t}^{1}$ $+\cdots
+Q_{s,t}^{\lfloor q\rfloor}+\widetilde{Q}_{s,t}^{\lfloor q\rfloor+1}+$
$\cdots+\widetilde{Q}_{s,t}^{m}$ $\ $on $\mathcal{P}_{d_{2},m}\left(
V\right)  $ such that $\left(  \ref{holder bound}\right)  $ holds for all
$j\in\left\{  1,\ldots,m\right\}  ,$ $\left(  s,t\right)  \in\Delta_{\left[
0,1\right]  }^{2}$. Moreover if $\overline{Q}_{s,t}^{j}$ is another
multiplicative extension such that $\left\Vert \overline{Q}_{s,t}%
^{j}\right\Vert \leq C\left(  j\right)  \left(  \left\vert t-s\right\vert
\right)  ^{j/q}$ for all $\left(  s,t\right)  \in\Delta_{\left[  0,1\right]
}^{2}$, then $\overline{Q}_{s,t}^{j}=\widetilde{Q}_{s,t}^{j}$ for all
$j\in\left\{  1,\ldots,m\right\}  $.
\end{lemma}

Before we begin the proof of the lemma we recall the neo-classical inequality
from \cite{L} (Lemma 2.2.2).

\begin{theorem}
[Neo-classical inequality, Lyons 98]For any $q\in\lbrack1,\infty),$
$n\in\mathbb{N}$ and $s,t\geq0$%
\[
\frac{1}{q^{2}}\sum_{i=0}^{n}\frac{s^{\frac{i}{q}}t^{\frac{n-i}{q}}}{\left(
\frac{i}{q}\right)  !\left(  \frac{n-i}{q}\right)  !}\leq\frac{\left(
s+t\right)  ^{n/q}}{\left(  n/q\right)  !}.
\]

\end{theorem}

\begin{proof}
[Proof of Lemma \ref{holder control}]\bigskip We will inductively construct
$Q_{s,t}^{\left[  n\right]  }$ for $n>\lfloor q\rfloor$ , the base case of the
induction following from the assumption on the $Q_{s,t}^{j},j=1,\ldots,\lfloor
q\rfloor$. The proof closely follows the proof of the classical extension
theorem for rough paths (see \cite{LCL} p.45f). To extend from $n-1\geq\lfloor
q\rfloor$ to $n$ first let on $\mathcal{P}_{d_{2},n}\left(  V\right)  $
\[
\widehat{Q}_{s,t}:=\sum_{j=1}^{n-1}Q_{s,t}^{j}.
\]
Given any finite partition $\mathcal{D}$ of the interval $\left[  s,t\right]
$ define $Q_{s,t}^{\left[  n\right]  ,\mathcal{D}}$ by setting%
\[
Q_{s,t}^{\left[  n\right]  ,\mathcal{D}}:=%
%TCIMACRO{\dprod \limits_{\mathcal{D}}}%
%BeginExpansion
{\displaystyle\prod\limits_{\mathcal{D}}}
%EndExpansion
\widehat{Q}_{t_{i},t_{i+1}}.
\]
By the pigeon hole principle there is $t_{j}$ such that
\[
\left(  t_{j+1}-t_{j-1}\right)  \leq\frac{2}{\left\vert \mathcal{D}\right\vert
-1}\left(  t-s\right)
\]
and we may coarsen the partition by dropping $t_{j}$ and write $\mathcal{D}%
^{\prime}:=\mathcal{D\setminus}\left\{  t_{j}\right\}  .$ Then
\[
Q_{s,t}^{\left[  n\right]  ,\mathcal{D}}-Q_{s,t}^{\left[  n\right]
,\mathcal{D}^{\prime}}=\widehat{Q}_{s,t_{1}}\cdots\left(  \widehat{Q}%
_{t_{j-1},t_{j}}\widehat{Q}_{t_{j},t_{j+1}}-\widehat{Q}_{t_{j-1},t_{j+1}%
}\right)  \cdots\widehat{Q}_{t_{\left\vert \mathcal{D}\right\vert -1},t}%
\]
and noting that $\widehat{Q}_{t_{j-1},t_{j}}\widehat{Q}_{t_{j},t_{j+1}%
}-\widehat{Q}_{t_{j-1},t_{j+1}}$ is a homogeneous polynomial of degree $n$ we
see that%
\[
Q_{s,t}^{\left[  n\right]  ,\mathcal{D}}-Q_{s,t}^{\left[  n\right]
,\mathcal{D}^{\prime}}=\sum_{i=1}^{n-1}Q_{t_{j-1},t_{j}}^{i}Q_{t_{j},t_{j+1}%
}^{n-i}.
\]
Therefore {}using the submultiplicative property for the norm, the inductive
hypothesis and finally the neo-classical inequality we see that%
\begin{align}
\left\Vert \pi_{n}\left(  Q_{s,t}^{\left[  n\right]  ,\mathcal{D}}%
-Q_{s,t}^{\left[  n\right]  ,\mathcal{D}^{\prime}}\right)  \right\Vert  &
=\left\Vert \sum_{i=1}^{n-1}Q_{t_{j-1},t_{j}}^{i}Q_{t_{j},t_{j+1}}%
^{n-i}\right\Vert \leq\sum_{i=1}^{n-1}\left\Vert Q_{t_{j-1},t_{j}}%
^{i}\right\Vert \left\Vert Q_{t_{j},t_{j+1}}^{n-i}\right\Vert
\label{will be modified}\\
&  \leq\sum_{i=1}^{n-1}\left(  \frac{\left(  C\left\vert t_{j}-t_{j-1}%
\right\vert \right)  ^{i/q}}{\theta\left(  i/q\right)  !}\right)  \left(
\frac{\left(  C\left\vert t_{j+1}-t_{j}\right\vert \right)  ^{\left(
n-i\right)  /q}}{\theta\left(  \left(  n-i\right)  /q\right)  !}\right)
\nonumber\\
&  \leq\frac{q^{2}}{\theta}\left(  \frac{2}{\left\vert \mathcal{D}\right\vert
-1}\right)  ^{\frac{n}{q}}\frac{\left(  C\left\vert t-s\right\vert \right)
^{\frac{n}{q}}}{\theta\left(  n/q\right)  !}.\nonumber
\end{align}
Successively dropping points from the partition until $\mathcal{D=}\left\{
s,t\right\}  $ we see that%
\[
\left\Vert \pi_{n}\left(  Q_{s,t}^{\left[  n\right]  ,\mathcal{D}}%
-\widehat{Q}_{s,t}^{{}}\right)  \right\Vert \leq\frac{q^{2}}{\theta}\left(
1+2^{n/q}\left(  \zeta\left(  \frac{\lfloor q\rfloor+1}{q}\right)  -1\right)
\right)  \frac{\left(  C\left\vert t-s\right\vert \right)  ^{\frac{n}{q}}%
}{\theta\left(  n/q\right)  !}.
\]
Thus whenever $\theta\geq q^{2}\left(  1+2^{n/q}\left(  \zeta\left(
\frac{\lfloor q\rfloor+1}{q}\right)  -1\right)  \right)  $ the maximal
inequality implies that%
\[
\left\Vert \pi_{n}\left(  Q_{s,t}^{\left[  n\right]  ,\mathcal{D}}\right)
^{n}\right\Vert \leq\frac{\left\vert t-s\right\vert ^{\frac{n}{q}}}%
{\theta\left(  n/q\right)  !}%
\]
holds for any partition of $\left[  s,t\right]  .$ It remains to verify the
existence of the limit $\lim_{\left\vert \mathcal{D}\right\vert \rightarrow
0}Q_{s,t}^{n,\mathcal{D}}.$ We proceed as in \cite{LCL} and exhibit the Cauchy
property for the sequence. Suppose $\mathcal{D=}\left(  t_{j}\right)  $ and
$\widetilde{\mathcal{D}}$ are two partitions of mesh size less than $\delta.$
Let $\widehat{\mathcal{D}}$ denote the common refinement of the two partitions
and let $\widehat{\mathcal{D}}_{j}=\left[  t_{j},t_{j+1}\right]
\cap\widetilde{\mathcal{D}}$ . Then%
\[
Q_{s,t}^{n,\widehat{\mathcal{D}}}-Q_{s,t}^{n,\mathcal{D}}=\sum Q_{t_{0},t_{1}%
}^{n,\widehat{\mathcal{D}}_{0}}\dots Q_{t_{j-1},t_{j}}^{n,\widehat{\mathcal{D}%
}_{j-1}}\left(  Q_{t_{j},t_{j+1}}^{n,\widehat{\mathcal{D}}_{j}}-\widehat{Q}%
_{t_{j},t_{j+1}}\right)  \dots Q_{t_{\left\vert \mathcal{D}\right\vert -1}%
,t}^{n,\widehat{\mathcal{D}}_{j}}.
\]
As seen before this is a sum of homogeneous polynomials of degree $n$ and by
the maximal inequality%
\[
\left\Vert \pi_{n}\left(  Q_{s,t}^{n,\widehat{\mathcal{D}}}-Q_{s,t}%
^{n,\mathcal{D}}\right)  \right\Vert \leq\sum_{\mathcal{D}}\frac{\left\vert
t_{j+1}-t_{j}\right\vert ^{\frac{n}{q}}}{\theta\left(  n/q\right)  !}\leq
\frac{\left\vert t-s\right\vert }{\theta\left(  n/q\right)  !}\delta^{\frac
{n}{q}-1}%
\]
as $\frac{n}{q}-1>0$ we have a uniform estimate in $\delta$ independent of the
choice of partition. Going through the same argument for the partition
$\widetilde{\mathcal{D}}$ and using the triangle inequality the Cauchy
property is established and the existence of the limit follows. The uniqueness
of the limit follows as in \cite{LCL}. The difference of two multiplicative
functionals that agree up to level $\lfloor q\rfloor$ is additive (see Lyons
\cite{L} Lemma 2.2.3) As the difference of the extensions is also a continuous
path and by assumption
\[
\left\Vert \overline{Q}_{s,t}^{\lfloor q\rfloor+1}-\widetilde{Q}%
_{s,t}^{\lfloor q\rfloor+1}\right\Vert \leq C\left(  \lfloor q\rfloor
+1\right)  \left\vert t-s\right\vert ^{\frac{\lfloor q\rfloor+1}{q}}%
\]
it follows that $\overline{Q}_{s,t}^{\lfloor q\rfloor+1}-\widetilde{Q}%
_{s,t}^{\lfloor q\rfloor+1}$ is identically zero. A simple induction now
completes the proof.
\end{proof}

\begin{lemma}
\label{main bound cor}For any $1/3<\gamma<1/2$ there exist a constant
$\theta>0$ and random variables $c(\gamma,\omega)$, almost surely finite, such
that
\begin{equation}
\left\Vert R_{s,t}^{n,\bar{\imath}}\right\Vert _{H^{1}\rightarrow H^{1}}%
\leq\frac{\left(  c\left(  \gamma,\omega\right)  \left\vert t-s\right\vert
\right)  ^{n\gamma}}{\theta\left(  n\gamma\right)  !}.
\label{final a priory estimates}%
\end{equation}
and
\begin{equation}
\left\Vert R_{s,t}^{n,\bar{\imath}}\right\Vert _{H^{-1}\rightarrow H^{-1}}%
\leq\frac{\left(  c\left(  \gamma,\omega\right)  \left\vert t-s\right\vert
\right)  ^{n\gamma}}{\theta\left(  n\gamma\right)  !}
\label{final a priory estimates 2}%
\end{equation}
for all $\bar{\imath}\in S\left(  n\right)  ,n\in\mathbb{N},$ $0<s<t\leq1$.
\end{lemma}

\begin{proof}
We now take for $V$ the space of bounded linear operators on (the completion
of) $H^{1}$ and $H^{-1}$ respectively. From the a priori estimates we know
that $Q_{s,t}^{\left[  n\right]  }\in$ $\mathcal{P}_{d_{2},n}\left(  V\right)
$ for all $n\geq1.$ First note that by Lemma \ref{a priori estimate} $\left(
Q_{s,t}^{1},Q_{s,t}^{2}\right)  $ satisfies the assumptions of Proposition
\ref{holder control} with $3>q=1/\gamma$ and therefore has a multiplicative
extension $\widetilde{Q}_{s,t}^{j}$ controlled in the sense of $\left(
\ref{holder bound}\right)  .$ Once again by Lemma \ref{a priori estimate} the
uniqueness part of Proposition \ref{holder control} applies and we deduce that
$Q_{s,t}^{j}=\widetilde{Q}_{s,t}^{j}$ for $j\in\mathbb{N}$.
\end{proof}

Armed with these two factorially decaying a priori estimates we are finally
ready proof a regularity estimate for $Q^{n}$ that decays factorially in $n.$
When considering $R_{s,t}^{n,\bar{\imath}},\bar{\imath}\in S\left(  n\right)
$ as an operator from $H^{-1}$ to $H^{1}$ we cannot directly apply Lemma
\ref{holder control} as the a priori bounds in Lemma \ref{a priori estimate}
have singularities for small $n$. Instead we exploit that there is more than
one way to estimate the operator norm of the composition of such operators.
Together with the estimates already obtained in Lemma \ref{main bound cor}
this will be sufficient to proof factorially decaying bounds for $n$
sufficiently large. We recall Proposition \ref{main factorial bound} and
restate it in the notation of the current section.

\textbf{Proposition \ref{main factorial bound}}: Let $1/3<\gamma<1/2$ be
fixed. \textit{There exists }$\theta>0,$\textit{\ }$\gamma^{\prime}\in\left(
1/3,\gamma\right)  $\textit{, }$m_{0}\in\mathbb{N}~\ $and\textit{\ random
variables }$c(\gamma^{\prime},\omega)$\textit{, almost surely finite, such
that }%
\begin{equation}
\left\Vert R_{s,t}^{n,\bar{\imath}}\right\Vert _{H^{-1}\rightarrow H^{1}}%
\leq\frac{\left(  c\left(  \gamma^{\prime},\omega\right)  \left\vert
t-s\right\vert \right)  ^{n\gamma^{\prime}}}{\theta\left(  n\gamma^{\prime
}\right)  !} \label{eq main factorial bound}%
\end{equation}
\textit{for all }$n\geq m_{0}.$

Before we begin the proof note that by choosing $\gamma^{\prime}<\gamma$ we
have for $n$ sufficiently large by Lemma \ref{a priori estimate}%
\[
\left\Vert R_{s,t}^{n,\bar{\imath}}\right\Vert _{H^{-1}\rightarrow H^{1}}\leq
c\left(  \gamma,n,\omega\right)  \left\vert t-s\right\vert ^{n\gamma-2\ell
}\leq c\left(  \gamma,n,\omega\right)  \left\vert t-s\right\vert
^{n\gamma^{\prime}}%
\]
for all \thinspace$0<s<t<1.$

\begin{proof}
[Proof of Proposition \ref{main factorial bound}]Choose $m_{0}$ and
$0<\gamma^{\prime}\leq\gamma$ such that $\gamma n-\ell\geq\gamma^{\prime}n$
for all $n\geq m_{0}.$ Using Corollary \ref{main bound cor} and Lemma
\ref{a priori estimate} (with $\gamma=\gamma^{\prime})$ we can find $c\left(
\gamma^{\prime},\omega\right)  $ such that simultaneously $\left(
\ref{eq main factorial bound}\right)  $ holds for all $n\in\left[
m_{0},2m_{0}\right]  $ and the two inequalities $\left(
\ref{final a priory estimates}\right)  $ and $\left(
\ref{final a priory estimates 2}\right)  $ hold for all $n\in\mathbb{N}$. Note
that this also serves as the base case for our induction argument. For this
lemma we set $V$ to be the space of bounded linear operators from $H^{-1}$ to
$H^{1}$.

We argue now exactly as in the proof Lemma \ref{holder control} to extend the
functional from level $n\geq2m_{0}$ to $n+1,$ with the only difference being
that we have no direct control over $\left\Vert R_{s,t}^{k,\bar{\imath}%
}\right\Vert _{H^{-1}\rightarrow H^{1}}$ for $\bar{\imath}\in S\left(
k\right)  ,$ $k<m_{0}.$ We therefore replace inequality $\left(
\ref{will be modified}\right)  $ with the following more refined estimate that
exploits that the operator norm of a composition of two operators can be
estimated in several ways, which allows us to draw on our a priori estimates
in Lemma \ref{a priori estimate}. We have
\begin{align}
&  \left\Vert \pi_{n+1}\left(  Q_{s,t}^{D,n+1}-Q_{s,t}^{D^{\prime}%
,n+1}\right)  \right\Vert \nonumber\\
&  \leq\sum_{i=1}^{n}\left\Vert Q_{t_{j-1},t_{j}}^{i}Q_{t_{j},t_{j+1}}%
^{n+1-i}\right\Vert \nonumber\\
&  =\sum_{i=1}^{n}\left\Vert \sum_{\bar{m}\in S\left(  i\right)  }\sum
_{\bar{l}\in S\left(  n+1-i\right)  }R_{t_{j-1},t_{j}}^{i,\bar{m}}%
R_{t_{j},t_{j+1}}^{n+1-i,\bar{l}}\varepsilon_{\bar{m}\ast\bar{l}}\right\Vert
\nonumber\\
&  =\sum_{i=1}^{n}\sup_{\bar{m}\in S\left(  i\right)  ,\bar{l}\in S\left(
n+1-i\right)  }\left\Vert R_{t_{j-1},t_{j}}^{i,\bar{m}}R_{t_{j},t_{j+1}%
}^{n+1-i,\bar{l}}\right\Vert _{H^{-1}\rightarrow H^{1}}\nonumber\\
&  \leq\sum_{i=1}^{m_{0}-1}\sup_{\bar{m}\in S\left(  i\right)  }\left\Vert
R_{t_{j-1},t_{j}}^{i,\bar{m}}\right\Vert _{H^{1}\rightarrow H^{1}}\sup
_{\bar{l}\in S\left(  n+1-i\right)  }\left\Vert R_{t_{j},t_{j+1}}%
^{n+1-i,\bar{l}}\right\Vert _{H^{-1}\rightarrow H^{1}}\nonumber\\
&  +\sum_{i=m_{0}}^{n}\sup_{\bar{m}\in S\left(  i\right)  }\left\Vert
R_{t_{j-1},t_{j}}^{i,\bar{m}}\right\Vert _{H^{-1}\rightarrow H^{1}}\sup
_{\bar{l}\in S\left(  n+1-i\right)  }\left\Vert R_{t_{j},t_{j+1}}%
^{n+1-i,\bar{l}}\right\Vert _{H^{-1}\rightarrow H^{-1}}\nonumber\\
&  \leq\sum_{j=1}^{n}\frac{\left(  C\left\vert t-s\right\vert \right)
^{\gamma^{\prime}j}}{\theta\left(  j\gamma^{\prime}\right)  !}\frac{\left(
C\left\vert t-s\right\vert \right)  ^{\gamma^{\prime}m+1-j}}{\theta\left(
\left(  m+1-j\right)  \gamma^{\prime}\right)  !} \label{replacement estimate}%
\end{align}
The bounds for $\left\Vert R_{t_{j-1},t_{j}}^{i,\bar{m}}\right\Vert
_{H^{1}\rightarrow H^{1}}$ and $\left\Vert R_{t_{j},t_{j+1}}^{n+1-i,\bar{l}%
}\right\Vert _{H^{-1}\rightarrow H^{-1}}$ use inequalities $\left(
\ref{final a priory estimates}\right)  $and $\left(
\ref{final a priory estimates 2}\right)  $ respectively. The bounds for
$\sup_{\bar{l}\in S\left(  n+1-i\right)  }\left\Vert R_{t_{j},t_{j+1}%
}^{n+1-i,\bar{l}}\right\Vert _{H^{-1}\rightarrow H^{1}}$ and $\ \sup_{\bar
{m}\in S\left(  i\right)  }\left\Vert R_{t_{j-1},t_{j}}^{i,\bar{m}}\right\Vert
_{H^{-1}\rightarrow H}$ follow (for the appropriate values of $i$) from the
inductive hypothesis$.$ With this modification in place arguing exactly as in
the proof of Lemma \ref{holder control} yields the result. Note that the
extension is only carried out for $n\geq2m_{0}.$ For $m_{0}\leq n<2m_{0}$ the
estimates use the a priori bounds.
\end{proof}

\begin{remark}
\label{final duality remark}To extend the proof of Proposition
\ref{main factorial bound} to cover the terms in the expansion of $\rho
_{t}^{\ast}$ we make the following modifications. In place of $R_{s,t}%
^{n,\bar{\imath}}$ we have
\[
X_{s,t}^{n,\bar{\imath}}=\int_{\Delta_{s,t}^{k}}P_{t-t_{n}}H_{i_{n}}%
P_{t_{n}-t_{n-1}}H_{i_{n-1}}\cdots H_{i_{1}}P_{t_{1}-s}dY_{t_{1}}^{i_{1}%
}\cdots dY_{t_{n}}^{i_{n}},
\]
i.e. the order of non-commutative product in the integrand is reversed. We
therefore define $\mathcal{\bar{P}}_{d_{2},k}\left(  V\right)  $ as
$\mathcal{P}_{d_{2},k}\left(  V\right)  $ but with the multiplication in
$\left(  \ref{mult def}\right)  $ replaced by%
\begin{equation}
ab:=\sum_{v=0}^{k}\sum_{j=0}^{v}\sum_{\bar{\imath}\in S\left(  j\right)  }%
\sum_{\bar{l}\in S\left(  v-j\right)  }b_{\bar{l}}a_{\bar{\imath}}%
\varepsilon_{\bar{\imath}\ast\bar{l}}. \label{mod mult}%
\end{equation}
With this modification $\left(  \ref{mult property prep}\right)  $ becomes%
\begin{align*}
X_{s,t}^{k,\bar{\imath}}  &  =\int_{\Delta_{s,u}^{k}}P_{t-t_{k}}H_{i_{k}%
}P_{t_{k}-t_{k-1}}H_{i_{k-1}}\cdots H_{i_{1}}P_{t_{1}-s}dY_{t_{1}}^{i_{1}%
}\cdots dY_{t_{k}}^{i_{k}}\\
&  +\int_{\Delta_{u,t}^{k}}P_{t-t_{k}}H_{i_{k}}P_{t_{k}-t_{k-1}}H_{i_{k-1}%
}\cdots H_{i_{1}}P_{t_{1}-s}dY_{t_{1}}^{i_{1}}\cdots dY_{t_{k}}^{i_{k}}\\
&  +\sum_{j=1}^{k-1}\int_{\Delta_{u,t}^{k-j}}P_{t-t_{k}}H_{i_{k}}%
P_{t_{k}-t_{k-1}}\cdots P_{t_{j+2}-t_{j+1}}H_{i_{j+1}}P_{t_{j+1}-u}%
dY_{t_{j+1}}^{i_{j+1}}\cdots dY_{t_{k}}^{i_{k}}\\
&  +\int_{\Delta_{s,u}^{k}}P_{u-t_{j}}H_{i_{j}}P_{t_{j}-t_{j-1}}\cdots
H_{i_{1}}P_{t_{1}-s}dY_{t_{1}}^{i_{1}}\cdots dY_{t_{j}}^{i_{j}}\\
&  =\sum_{j=0}^{k}X_{u,t}^{k-j,\left(  i_{j+1},\ldots i_{k}\right)  }%
X_{s,u}^{j,\left(  i_{1},\ldots i_{j}\right)  }\\
&  =\sum_{\bar{m}\ast\bar{l}=\bar{\imath}}X_{u,t}^{\left\vert \bar
{l}\right\vert ,\bar{l}}X_{s,u}^{\left\vert \bar{m}\right\vert ,\bar{m}}.
\end{align*}
Combining this identity with the modified multiplication (\ref{mod mult}) we
see that $\left(  \ref{multiplicative property}\right)  $ holds on
$\mathcal{\bar{P}}_{d_{2},k}\left(  V\right)  ,$ i.e. our functional $\bar
{Q}_{s,t}^{\left[  n\right]  }=\sum_{j=0}^{n}\sum_{\bar{\imath}\in S\left(
j\right)  }X_{s,t}^{j,\bar{\imath}}\varepsilon_{\bar{\imath}}$ has the
multiplicative property. Going through the same steps as before with these
modifications in place the proof of Proposition \ref{main factorial bound} may
now be completed.
\end{remark}

\section{Appendix:}

\subsection{\bigskip Proof of the inequality (\ref{massbound})}

As $Z_{t}^{x}\geq0$, we have, by Jensen's inequality, that $\mathbb{\tilde{E}%
}\left[  Z_{t}^{x}|Y_{\cdot}^{x}\right]  ^{-1}\leq\mathbb{\tilde{E}}\left[
(Z_{t}^{x})^{-1}|Y_{\cdot}^{x}\right]  $. Then observe that, by integration by
parts
\begin{align*}
-\sum_{i=1}^{d_{2}}\int_{0}^{t}h^{i}(X_{s}^{x})\,dY_{s}^{x,i}  &  =\sum
_{i=1}^{d_{2}}\left(  -h^{i}(X_{t}^{x})\,Y_{t}^{x,i}\,+\int_{0}^{t}Y_{s}%
^{x,i}Ah^{i}(X_{s}^{x})ds+\sum_{j=1}^{d}\int_{0}^{t}Y_{s}^{x,i}V^{j}%
h^{i}(X_{s}^{x})\,dB_{s}^{j}\right) \\
&  \leq\sum_{i=1}^{d_{2}}\sup_{s\in\left[  0,t\right]  }\left\vert Y_{s}%
^{x,i}\right\vert \left(  \left\vert \left\vert h^{i}\right\vert \right\vert
_{\infty}+t\left\vert \left\vert Ah^{i}\right\vert \right\vert _{\infty
}\right) \\
&  +\frac{d_{2}t}{2}\sum_{i=1}^{d_{2}}\sup_{s\in\left[  0,t\right]
}\left\vert Y_{s}^{x,i}\right\vert ^{2}\left(  \sum_{j=1}^{d}\left\vert
\left\vert V^{j}h^{i}\right\vert \right\vert ^{2}\right)  +\eta_{t}^{x},
\end{align*}
where
\[
\eta_{t}^{x}=\sum_{j=1}^{d}\int_{0}^{t}\left(  \sum_{i=1}^{d_{2}}Y_{s}%
^{x,i}V^{j}h^{i}(X_{s}^{x})\right)  \,dB_{s}^{j}-\sum_{j=1}^{d}\frac{1}{2}%
\int_{0}^{t}\left(  \sum_{i=1}^{d_{2}}Y_{s}^{x,i}V^{j}h^{i}(X_{s}%
^{x})\,\right)  ^{2}ds.
\]
Since, $\mathbb{\tilde{E}}\left[  \exp\left.  \eta_{t}^{x}\right\vert
\mathcal{Y}_{t}^{x}\right]  =1$,we get that%
\[
\left(  1/\rho_{t}^{x}\left(  1\right)  \right)  <\exp C\left(  \sum
_{i=1}^{d_{2}}\sup_{s\in\left[  0,t\right]  }\left\vert Y_{s}^{x,i}\right\vert
+\sup_{s\in\left[  0,t\right]  }\left\vert Y_{s}^{x,i}\right\vert ^{2}\right)
,
\]
where $C$ is a constant independent of $x$, $C=\max_{i=1,...,d}(\left\vert
\left\vert h^{i}\right\vert \right\vert _{\infty}+t\left\vert \left\vert
Ah^{i}\right\vert \right\vert _{\infty}+\frac{d_{2}t}{2}\sum_{j=1}%
^{d}\left\vert \left\vert V^{j}h^{i}\right\vert \right\vert ^{2}).$ Inequality
(\ref{massbound}) follows as $\sup_{x\in R^{N}}\sup_{s\in\left[  0,t\right]
}\left\vert Y_{s}^{x,i}\right\vert $ is finite almost surely.

\subsection{The expansion of the first three iterated integrals}

For the first integral we can express it using the following two terms:%
\begin{align*}
\int_{0}^{t_{2}}R_{\left(  t_{1},t_{2}\right)  }(\varphi)dY_{t_{1}}  &
=\int_{0}^{t_{2}}P_{t_{1}}\left(  hP_{t_{2}-t_{1}}(\varphi)\right)  dY_{t_{1}%
}\\
&  =q_{t_{2}}^{1}\left(  Y\right)  P_{t_{2}}\left(  h\varphi\right)  -\int%
_{0}^{t_{2}}q_{t_{1}}^{1}\left(  Y\right)  P_{t_{1}}\left(  \Psi_{1}%
P_{t_{2}-t_{1}}(\varphi)\right)  dt_{1}.
\end{align*}
For the second iterated integral we end up with the following five terms
(2+3)
\begin{align*}
\int_{0}^{t_{3}}\int_{0}^{t_{2}}R_{\left(  t_{1},t_{2},t_{3}\right)  }%
(\varphi)dY_{t_{1}}dY_{t_{2}}  &  =\int_{0}^{t_{3}}\int_{0}^{t_{2}}R_{\left(
t_{1},t_{2}\right)  }\left(  hP_{t_{3}-t_{2}}(\varphi)\right)  dY_{t_{1}%
}dY_{t_{2}}\\
&  =\int_{0}^{t_{3}}q_{t_{2}}^{1}\left(  Y\right)  P_{t_{2}}\left(
h^{2}P_{t_{3}-t_{2}}(\varphi)\right)  dY_{t_{2}}\\
&  -\int_{0}^{t_{3}}\int_{0}^{t_{2}}q_{t_{1}}^{1}\left(  Y\right)  P_{t_{1}%
}\left(  \Psi_{1}P_{t_{2}-t_{1}}(hP_{t_{3}-t_{2}}(\varphi))\right)
dt_{1}dY_{t_{2}}\\
&  =q_{t_{3}}^{2}\left(  Y\right)  P_{t_{3}}\left(  h^{2}\varphi\right)
-\int_{0}^{t_{3}}q_{t_{2}}^{2}\left(  Y\right)  P_{t_{2}}\left(  \Psi
_{2}P_{t_{3}-t_{2}}(\varphi)\right)  dt_{2}\\
&  -q_{t_{3}}^{1}\left(  Y\right)  \int_{0}^{t_{3}}q_{t_{2}}^{1}\left(
Y\right)  P_{t_{2}}\left(  \Psi_{1}P_{t_{3}-t_{2}}(h\varphi))\right)  dt_{2}\\
&  +\int_{0}^{t_{3}}q_{t_{2}}^{1}\left(  Y\right)  ^{2}P_{t_{2}}\left(
\Psi_{1}(hP_{t_{3}-t_{2}}(\varphi))\right)  dt_{2}\\
&  +\int_{0}^{t_{3}}q_{t_{2}}^{1}\left(  Y\right)  \int_{0}^{t_{2}}q_{t_{1}%
}^{1}\left(  Y\right)  P_{t_{1}}\left(  \Psi_{1}P_{t_{2}-t_{1}}(\Psi
_{1}P_{t_{3}-t_{2}}(\varphi))\right)  dt_{1}dt_{2}%
\end{align*}
For the third iterated integral we end up with the following 14 terms (2+3x4)%
\begin{align*}
\int_{0}^{t_{4}}\int_{0}^{t_{3}}\int_{0}^{t_{2}}R_{\left(  t_{1},t_{2}%
,t_{3}\right)  }(\varphi)dY_{t_{1}}dY_{t_{2}}dY_{t_{3}}  &  =\int_{0}^{t_{4}%
}q_{t_{3}}^{2}\left(  Y\right)  P_{t_{3}}\left(  h^{3}P_{t_{4}-t_{3}}%
(\varphi)\right)  dY_{t_{3}}\\
&  -\int_{0}^{t_{4}}\int_{0}^{t_{3}}q_{t_{2}}^{2}\left(  Y\right)  P_{t_{2}%
}\left(  \Psi_{2}P_{t_{3}-t_{2}}(hP_{t_{4}-t_{3}}(\varphi))\right)
dt_{2}dY_{t_{3}}\\
&  -\int_{0}^{t_{4}}q_{t_{3}}^{1}\left(  Y\right)  \int_{0}^{t_{3}}q_{t_{2}%
}^{1}\left(  Y\right)  P_{t_{2}}\left(  \Psi_{1}P_{t_{3}-t_{2}}(h^{2}%
P_{t_{4}-t_{3}}(\varphi)))\right)  dt_{2}dY_{t_{3}}\\
&  +\int_{0}^{t_{4}}\int_{0}^{t_{3}}q_{t_{2}}^{1}\left(  Y\right)
^{2}P_{t_{2}}\left(  \Psi_{1}(hP_{t_{3}-t_{2}}(hP_{t_{4}-t_{3}}(\varphi
)))\right)  dt_{2}dY_{t_{3}}\\
&  +\int_{0}^{t_{4}}\int_{0}^{t_{3}}q_{t_{2}}^{1}\left(  Y\right)  \int%
_{0}^{t_{2}}q_{t_{1}}^{1}\left(  Y\right)  P_{t_{1}}\left(  \Psi_{1}%
P_{t_{2}-t_{1}}(\Psi_{1}P_{t_{3}-t_{2}}(hP_{t_{4}-t_{3}}(\varphi)))\right)
dt_{1}dt_{2}dY_{t_{3}}\\
&  =q_{t_{4}}^{3}\left(  Y\right)  P_{t_{4}}\left(  h^{3}\varphi\right)
-\int_{0}^{t_{4}}q_{t_{3}}^{2}\left(  Y\right)  P_{t_{3}}\left(  \Psi
_{3}P_{t_{4}-t_{3}}(\varphi)\right)  dt_{3}\\
&  -q_{t_{4}}^{1}\left(  Y\right)  \int_{0}^{t_{4}}q_{t_{3}}^{2}\left(
Y\right)  P_{t_{3}}\left(  \Psi_{2}P_{t_{4}-t_{3}}(h\varphi)\right)  dt_{3}\\
&  +....
\end{align*}

\end{document}